\documentclass[11pt]{amsart}
\usepackage{pdfsync}
\usepackage[linktocpage]{hyperref}
\usepackage{amssymb, paralist, xspace, graphicx, url, amscd, euscript, mathrsfs,stmaryrd,epic,eepic,color}
\usepackage[all]{xy}
\usepackage{amsthm}
\usepackage{enumerate}
\usepackage{tikz}
\usepackage{tikz-cd}
\usetikzlibrary{calc}
\usetikzlibrary{arrows}
\usetikzlibrary{matrix}
\usepackage{lipsum}
\SelectTips{cm}{}
\numberwithin{equation}{section}

1
\numberwithin{subsection}{section}

\allowdisplaybreaks[1]

\newtheorem*{namedtheorem}{\theoremname}
\newcommand{\theoremname}{testing}

\newtheorem{Theorem}{Theorem}
\newtheorem{theorem}[subsubsection]{Theorem}
\newtheorem{proposition}[subsubsection]{Proposition}
\newtheorem{conjecture}{Conjecture}
\newtheorem{corollary}[subsubsection]{Corollary}

\newtheorem{lemma}[subsubsection]{Lemma}
\theoremstyle{definition}
\newtheorem{definition}[subsubsection]{Definition}
\newtheorem{question}[subsection]{Question}

\newtheorem{example}[subsubsection]{Example}
\newtheorem{remark}[subsubsection]{Remark}

\theoremstyle{remark}

\newcommand\cA{\mathcal{A}}
\newcommand\cB{\mathcal{B}}

\newcommand\cF{\mathcal{F}}

\newcommand\cK{\mathcal{K}}
\newcommand\cL{\mathcal{L}}
\newcommand\cM{\mathcal{M}}
\newcommand\cN{\mathcal{N}}
\newcommand\cO{\mathcal{O}}
\newcommand\cP{\mathcal{P}}

\newcommand\cS{\mathcal{S}}
\newcommand\cT{\mathcal{T}}
\newcommand\cU{\mathcal{U}}

\renewcommand\AA{\mathbb{A}}

\newcommand\CC{\mathbb{C}}
\newcommand\C{\mathbb{C}}

\newcommand\EE{\mathbb{E}}

\newcommand\HH{\mathbb{H}}

\newcommand\PP{\mathbb{P}}
\newcommand\QQ{\mathbb{Q}}
\newcommand\RR{\mathbb{R}}
\renewcommand\SS{\mathbb{S}}

\newcommand\ZZ{\mathbb{Z}}

\newcommand\bE{\mathbf{E}}

\newcommand\bH{\mathbf{H}}

\newcommand\bx{\mathbf{x}}

\newcommand\frg{\mathfrak{g}}

\newcommand\NL{{\rm NL}}

\newcommand\rank{{\rm rank}}
\newcommand\Pic{{\rm Pic}}

\newcommand{\SO}{{\rm SO}}

\newcommand{\Aut}{{\rm Aut}}

\newcommand{\bv}{{\bf v}}

\newcommand{\Mp}{{\rm Mp}}
\newcommand{\Sp}{{\rm Sp}}
\newcommand{\GL}{{\rm GL}}

\newcommand{\Gspin}{{\rm GSpin}}

\newcommand{\CH}{{\rm CH}}
\newcommand{\BV}{{\rm BV}}
\newcommand{\Mon}{{\rm Mon}}
\newcommand{\td}{{\rm td}}
\newcommand{\SC}{{\rm SC}}
\theoremstyle{plain}
\theoremstyle{definition}

\begin{document}
	
\title[Tautological classes on moduli of hyperk\"ahler manifolds]{Tautological classes on moduli space of hyperk\"ahler manifolds}
\author{Nicolas Bergeron} 

\address{Sorbonne Universit\'es, UPMC Univ Paris 06, Institut de Math\'ematiques de Jussieu--Paris Rive Gauche, UMR 7586, CNRS, Univ Paris Diderot, Sorbonne Paris Cit\'e, F-75005, Paris, France}

\email{nicolas.bergeron@imj-prg.fr}

\author{Zhiyuan Li}
\address{Shanghai Center for Mathematical Sciences \\ Fudan University\\
220 Handan Road, Shanghai, 200433 China\\
	}
\email{zhiyuan\_li@fudan.edu.cn}

\begin{abstract}In this paper, we discuss the cycle theory on moduli spaces $\cF_h$ of $h$-polarized hyperk\"ahler manifolds. Firstly, we construct the tautological ring on  $\cF_h$ following the work of Marian, Oprea and Pandharipande on the tautological conjecture on moduli spaces of K3 surfaces. We study the  tautological classes in cohomology groups and prove that most of them are linear combinations of Noether-Lefschetz cycle classes. In particular, we prove the cohomological version of the tautological conjecture on moduli space of K3$^{[n]}$-type hyperk\"ahler manifolds with $n\leq 2$. Secondly, we prove the cohomological generalized Franchetta conjecture on universal family of these hyperk\"ahler manifolds. 
\end{abstract}
\maketitle

\section{Introduction}

The tautological ring of the moduli space $\cM_g$  of genus $g\geq 2$ curves, originally studied by Mumford,  is  the subring of the Chow ring generated by $\kappa$-classes which appear most naturally in geometry. There have been substantial progress in understanding tautological ring of  moduli space of curves  in the past decades. In higher dimensional moduli theory, little is known even regarding definitions. Recently, there have been some developments towards the cycle theory on moduli spaces of K3 surfaces.  In this paper, we would like to investigate the tautological class problem on moduli spaces of K3 type varieties.

\subsection{MOP conjecture and its generalization}Let $\cK_g$ be the moduli space of primitively polarized K3 surfaces of genus $g$.   The cycle theory on $\cK_g$  appears to be much more complicated than on $\cM_g$ because there are many more classes in the Chow groups. 
First of all, there are so called {\it NL-cycles} on $\cK_g$ arising from the  Noether-Lefschetz theory.   For  $1\leq r\leq19$,  the {\it $r$-th higher  Noether-Lefschetz  locus} $$\cN^r(\cK_g) \subseteq \cK_g$$ parametrizing  the  K3 surfaces in  $\cK_g$ with Picard number greater than $r$, is a countable union of subvarieties of codimension $r$.  Each irreducible component of $\cN^r(\cK_g)$ occurs as the moduli space of certain lattice-polarized K3 surfaces in $\cK_g$  and  a NL-cycle of codimension $r$ is defined as some  linear combinations of such irreducible subvarieties.  They  play a very important role in the cycle theory of $\cK_g$.   It has been shown that the Picard group $\Pic_\QQ(\cK_g)$ with rational coefficients is generated by NL-divisors, see \cite{BLMM16}.  This motivates us to study the subring  generated by  all NL-cycles, denoted by  $\NL^\ast(\cK_g)$.

Secondly,  similar as the case in $\cM_g$, one can obtain natural cycle classes from the tautological bundles on the universal family. Let $\cK_\Sigma \subset \cK_g$ be the $r$-th higher Noether-Lefschetz locus corresponding to some Picard lattice $\Sigma$ of rank $r+1$.  In \cite{MOP15}, Marian, Oprea and Pandharipande (MOP) have introduced $\kappa$-classes $\kappa^\cB_{a_0,\ldots,a_r;b}$ on $\cK_\Sigma$, see Definition \ref{D321}. The {\it tautological ring}  $\mathrm{R}^\ast(\cK_g)\subseteq \CH_\QQ^\ast(\cK_g)$ is  the $\QQ$-subalgebra generated by the images of the $\kappa$-classes on $\cK_\Sigma$ via pushforward maps. The tautological conjecture on $\cK_g$ states:
 \begin{conjecture}[MOP] \label{ConjMOP}
 Let $\mathrm{R}^\ast(\cK_g)$ be the subring of $\CH^\ast_\QQ(\cK_g)$ generated by all the images of the $\kappa$-classes on $\cK_\Lambda$. 	Then $\NL^\ast(\cK_g)=  \mathrm{R}^\ast(\cK_g)$. 
 \end{conjecture}
 
More generally, as  polarized hyperk\"ahler manifolds are higher-dimensional generalizations of algebraic K3 surfaces and behave like K3 surfaces in many ways (cf.~\S 2), it is natural to study tautological rings on moduli spaces of all polarized hyperk\"ahler manifolds. Guided by the K3 case, we define kappa classes on smooth families of hyperk\"ahler manifolds as pushforwards of ``Beauville-Voisin classes'' on the generic fiber. These classes are named in reference to the work of Beauville and Voisin \cite{BV04} on weak splitting properties of algebraic hyperk\"ahler manifolds. In this way, we define tautological rings on moduli spaces of polarized hyperk\"ahler manifolds and formulate a generalization of Conjecture \ref{ConjMOP}   for  polarized hyperk\"ahler manifolds. We will refer to it as the {\it generalized  tautological conjecture}, see Conjecture \ref{GTC} for the details.

In this paper, we mainly consider the cohomological version of the generalized tautological conjecture and prove that  most  tautological classes in cohomology groups are in the span of Noether-Lefschetz classes. Our main results are Theorems \ref{main2} and \ref{tauthm}.  As  a consequence, we obtain

\begin{Theorem}\label{mainthm1}
	%When $r\leq 17$, 	every $\kappa$-class $[\kappa^\cB_{a_0,\ldots,a_r;b} ]$ in $H^{2m}(\cK_g ,\QQ)$ is in the span of special cycles of $\cK_g$, where $m=\sum_{i=0}^r a_i+2b-2$. In particular, 
	Let  $\mathrm{R}^\ast_{\rm hom}(\cK_g)\subseteq H^\ast(\cK_g,\QQ)$ be 
	the image of  $\mathrm{R}^\ast(\cK_g)$ in the cohomology ring, then $
	\NL^{\ast}_{\rm hom}(\cK_g)=\mathrm{R}^\ast_{\rm hom}(\cK_g ).
$ 
\end{Theorem}
Our approach is to show that all $\kappa$-classes on $\cK_\Sigma$ are in the span of Noether-Lefschetz cycles on $\cK_\Sigma$. However, this actually is not true when $r>17$ (see Remark~\ref{fail}).  Fortunately, this does not cause contradict Theorem \ref{mainthm1} because we will prove that $\mathrm{R}^{i}_{\rm hom}(\cK_g)=0$ for $i>17$. 
 
While we were writing this paper, the article \cite{PY16} appeared on arXiv. There Conjecture~\ref{ConjMOP} is essentially proved. They use Gromov-Witten theory. Beware however that they work with a slightly different (more canonical) definition of kappa classes.

\subsection{Generalized Franchetta conjecture} Denote by $\cK_{g}^\circ$ be the subspace of $\cK_g$ that consists of K3 surfaces with only trivial automorphisms. It carries a universal family $\pi : \cU_g^\circ\rightarrow \cK^\circ_g$.  Recall that, O'Grady's generalized Franchetta conjecture states
\begin{conjecture}[O'Grady]	\label{ConjGFC}
	For any $\alpha\in \CH^2(\cU_g^\circ)$,  the restriction $\alpha|_S $ to a closed fiber $S$ is a multiple of the Beauville-Voisin class $c_S$. 
\end{conjecture}
Here $c_S$ is a canonical class in $\CH_0 (S)$ represented by a point on a rational curve of $S$, which satisfies the following properties: 
\begin{enumerate}
	\item The intersection of two divisors on $S$ lies in $\mathbb{Z} c_S \subset \CH_0 (S)$.
	\item The second Chern class $c_2 (T_S )$ equals $24 c_S \in \CH_0 (S)$.
\end{enumerate}
Conjecture \ref{ConjGFC} has been confirmed when $g\leq 12$ (see \cite{PSY}) but remains widely open for large $g$.   
Let $\cT_\pi$ be the relative tangent bundle of $\pi$. According to \cite[Theorem 10.19]{Voi07}, Conjecture \ref{ConjGFC} is equivalent to the following 
\begin{conjecture}
 For any $\alpha\in \CH^2(\cU_g^\circ)$, there exists $m\in \QQ$ such that $\alpha-mc_2(\cT_\pi)$ supports on a proper subvariety of $\cK_g^\circ$.	
\end{conjecture}
 Our second result is a cohomological version of the generalized Franchetta conjecture (cf.~\cite{Vo12}).

\begin{theorem}\label{main3}
For any $\alpha\in \CH^2(\cU_g^\circ)$, there exists a rational number $m$ such that the class $[\alpha-mc_2(\cT_\pi)]\in H^4(\cU_g^\circ,\QQ)$ supports on Noether-Lefschetz divisors. In other words, $\alpha-mc_2(\cT_\pi)$ is cohomologically equivalent to zero on $\pi^{-1}(W)$ for some open subset $W\subseteq \cK_g^\circ$.
\end{theorem}

Here again, Theorem \ref{main3} follows from a more general theorem concerning the cohomological generalized Franchetta conjecture on hyperk\"ahler manifolds (see $\S$8.1).

\subsection{Organization of the paper}
In Sections 2 and 3, we review the  hyperk\"ahler geometry, especially the cycle theory and Torelli theorem.  The tautological ring on moduli spaces of polarized hyperk\"ahlers  is  defined in Section 4. In Section 5 and Section 6, we recap the work of \cite{BMM16} and \cite{BLMM16} about surjectivity results on special cycles of Shimura varieties of orthogonal type.  Following \cite{FM06} in Section 7 we construct a so-called Funke-Kudla-Millson ring in the space of differential forms with coefficients. Section 8 is devoted to the proof of the cohomological tautological conjecture and the generalized Franchetta conjecture. In the last section, we discuss the properties of the ring generated by special cycles on Shimura varieties. 

\subsection{Notation and conventions}
Throughout this paper, we write $\widehat{\ZZ}$ for the profinite completion of $\ZZ$. We denote by  $\AA$ the ad\`ele ring of  $\QQ$ and $\AA_f$ the ring of finite ad\`eles. If $G$ is a semisimple classical group over $\QQ$,  we let $L^2(G(\QQ)\backslash G(\AA))$ be the space of square integrable functions on  $G(\QQ)\backslash G(\AA)$ and   denote by $\mathcal{A} (G)$ (resp.~$\mathcal{A}_{\rm cusp} (G)$) the set of irreducible square integral (resp.~cuspidal) representations  occuring discretely in $L^2(G(\QQ)\backslash G(\AA))$.  For a complex variety $X$, from now on we shall use $\CH^\ast(X)$ to denote the Chow ring of $X$ with {\it rational} coefficients and denote by 
$\mathrm{DCH}^\ast(X)\subseteq \CH^\ast(X)$ the subring generated by divisor classes.

\subsection{Acknowledgements}We are very grateful to Brendan Hassett, Daniel Huybrechts and Eyal Markman for many helpful discussions and useful suggestions. The first named author would like to thank Claire Voisin. She first told him about O'Grady's generalized Franchetta conjecture and asked if the methods of \cite{BLMM16} could be used to prove the cohomological version of it. He would also like to thank her for many explanations related to these topics. The second named author would like to thank Panhdaripande  and Yin for a lot of useful conversations during his stay in ETH.

\section{Hyperk\"ahler manifolds and moduli}

\subsection{Basic theory of hyperk\"ahler manifolds} A smooth complex compact $2n$-dimensional manifold  $X$ is an {\it irreducible holomorphic symplectic} or {\it  hyperk\"ahler} manifold if it is simply connected and $H^0(X,\Omega^2_X)$ is spanned by an everywhere nondegenerate holomorphic
2-form $\omega_X$.  It carries an integral, primitive quadratic form $q_X$ on $H^2(X, \ZZ)$,  called the Beauville-Bogomolov (BB) form, which satisfies
\begin{enumerate}
	\item $q_X$ is non-degenerate and of signature $(3,b_2(X)-3)$
	\item  There exists a positive rational number $c$, {\it the  Fujiki invariant} such that $q_X^n(\alpha)=c\int_X \alpha^{2n}$ for all classes $\alpha\in H^2(X,\ZZ)$.
	\item The Hodge decomposition $H^2(X, \CC)=H^{2,0}(X)\oplus H^{1,1}(X)\oplus H^{0,2}(X)$ is orthogonal with respect to $q_X\otimes \CC$. 
\end{enumerate}
From $(2)$, the BB-form $q_X$ is deformation invariant and we will thus restrict to compact hyperk\"ahler manifolds of a fixed deformation class, {\it the isomorphism type}, say $\Lambda$, of the lattice realized by the BB-form  is unique.  We say an isometry $\varphi: H^2(X,\ZZ)\rightarrow \Lambda$ is a {\it marking} of $X$.  Moreover, if we consider the entire cohomology of $X$, we let $R$ denote an abstract ring such that the cohomology ring $H^\ast(X,\ZZ)$ is isomorphic to $R$. Then we say that an isomorphism $$\Phi:H^\ast(X,\ZZ)\xrightarrow{\sim}R,$$ is a {\it full marking} of $X$. If $X$ comes with an ample line bundle $H$,   we say that $(X,H)$  is a {\it polarized  hyperk\"ahler manifold}.

Such manifolds are quite rare to construct. The two  well-known  series of examples, found by Beauville \cite{Bea83}, are the Hilbert scheme of points on K3 surfaces and  generalized Kummer varieties (See Example 2.1.1). The only other known examples were constructed by O'Grady \cite{Grad99, Grad03}. See the details as below.

\begin{example}\label{example}
\begin{enumerate}[(i)]
	\item For $n>0$, the length $n$ Hilbert scheme $S^{[n]}$ of a K3 surface $S$ is a hyperk\"ahler manifold of dimension $2n$. The second cohomology $H^2(S^{[n]},\ZZ)$ under BB-form is an even lattice of signature $(3, 20)$ and it is isomorphic to
	\begin{equation}
	L_{n}=U^{\oplus 3}\oplus E_8(-1)^{\oplus 2}\oplus \left< -2(n-1)\right>
	\end{equation}
	where $U$ is the hyperbolic lattice of rank two and $E_8$ is the positive definite lattice associated to the Lie group of the same name. We say that a hyperk\"ahler manifold $X$ is of \emph{$\mathrm{K3}^{[n]}$-type} if it is deformation equivalent to $S^{[n]}$. In that case the Fujiki invariant is $c = (2n)!/(n!2n)$.
	
	\item 	If $A$ is an abelian surface and $s:A^{[n+1]}\rightarrow A$ is the morphism induced by the additive structure of $A$,  then the generalized Kummer variety $K_n(A):= s^{-1}(0)$ is a projective hyperk\"ahler manifold of dimension $2n$.   We say that  $X$ is of \emph{generalized Kummer type} if it is deformation equivalent to $K_{n}(A)$. In this case, $H^2(X,\ZZ)$ is isomorphic to 
		\begin{equation}
		L_{K,n}=U^{\oplus 3}\oplus \left< -2(n+1)\right>.
		\end{equation}
		while the Fujiki invariant is $c=(n + 1)(2n)!/(n!2n)$.
	
	\item {\bf (OG6)} Let  $S$ be a K3 surface, and $M$ the moduli
	space of stable rank 2 vector bundles on $S$, with Chern classes $c_1=0$, $c_2=4$.  It admits a natural
	compactification $\overline{M}$, obtained by adding classes of semi-stable torsion free sheaves.   There is a desingularization of $\overline{M}$ which is a hyperk\"ahler
	manifold of dimension 10. Its second cohomology is a lattice of signature $(3,21)$.
	
	\item {\bf (OG10)}  The similar construction can be done starting from
	rank 2 bundles with $c_1 = 0$, $c_2 = 2$ on an abelian surface, that gives a hyperk\"ahler
	manifold of dimension 6 as in (iii). Its  second cohomology is a lattice of signature $(3,5)$.
\end{enumerate}	
\end{example}

\subsection{Automorphisms and monodromy group}  
Let $X$ be a hyperk\"ahler manifold of type $\Lambda$ and real dimension $2n$. Let $\Aut(X)$ be the group of automorphisms of $X$. Its action on the cohomology gives a map
\begin{equation}\label{aut}
\mathrm{Aut}(X)\rightarrow \GL(H^\ast(X,\ZZ))
\end{equation}
which has finite kernel.  We say that $X$ is {\it  cohomologically rigidified} if $\Aut(X)$ acts faithfully on $H^\ast(X,\ZZ)$, i.e. \eqref{aut} is injective. So far, for all hyperk\"ahler manifolds in Example \ref{example},    they are known to be cohomologically rigidified except the case of  {\bf OG6} type (cf.~\cite{Bea83,Og12}).   In general, it remains open whether \eqref{aut} is  injective for all hyperk\"ahler manifolds. 

Moreover,  the automorphisms act on $H^2(X,\ZZ)$ preserving the BB-form. So projection to the second cohomology yields a map 
\begin{equation}\label{2aut}
\mathrm{Aut}(X)\rightarrow \mathrm{O}(H^2(X,\ZZ)).
\end{equation}
with  finite kernel.  According to \cite{HT13}, this  kernel is   deformation invariant  and  it is trivial if $\Lambda$ is of K3$^{[n]}$ or {\bf OG10} type and   \hbox{nontrivial} if $\Lambda$ is of generalized Kummer  or {\bf OG6} type.

 \subsection{~} The {\it  monodromy group} ${\rm Mon}(X)\subseteq \mathrm{GL}(H^\ast(X,\ZZ))$ is defined as the subgroup generated by monodromy representations of all connected families containing $X$. Let ${\rm Mon}^2(X)\subseteq \mathrm{Aut}(H^2(X,\ZZ)) $ be its image under the projection to the second cohomology group.  Verbitsky \cite{Ve13} has shown that the group ${\rm Mon}^2(X)$ is an arithmetic subgroup of $\mathrm{O}(H^2(X,\ZZ))$ and there is an exact sequence 
  \begin{equation}\label{monproj}
  1\rightarrow T\rightarrow {\rm Mon}(X)\rightarrow  {\rm Mon}^2(X) \rightarrow 1
  \end{equation}
  with $T$ a finite group. Moreover, $T$ is trivial if \eqref{2aut}  is injective (cf.~\cite[Corollary 7.3]{Ve13}). 
  
 Furthermore, given a polarized hyperk\"ahler manifold $(X,H)$,  the {\it polarized monodromy group} $${\rm Mon}^2(X,H)\subseteq {\rm Mon}^2(X)$$ is defined to be the stabilizer of $c_1(H)$. Then ${\rm Mon}^2(X,H)$ is isomorphic to an arithmetic subgroup  of $\mathrm{O}(H^2_{\rm prim}(X,\ZZ))$ (cf.~\cite[Theorem 3.4]{Ve13}) via a given marking. It has been shown by Markman \cite[Proposition 7.1]{Mar11} that the image is independent of the marking.
  
 %The proof of Proposition \ref{dec} relies on Markman's work on the monodromy groups of $\mathrm{K3}^{[n]}$-type hyperk\"ahler manifolds in \cite{Mar08}, where it has been shown that the Zariski closure $\overline{{\rm Mon}^2(X)}=G_X$ and  the quotient ${\rm Mon}(X)\rightarrow {\rm Mon}^2(X)$ has finite kernel.

\subsection{Group actions on cohomology} Let $G_X$ be the $\QQ$-algebraic group associated to $\SO(\Lambda )$. It acts naturally on $H^2(X,\QQ)$ 
through to the standard representation.  As $G_X$ is a connected component of the Zariski closure of $\Mon^2(X)$ and $\Mon(X)\rightarrow \Mon^2(X)$ has finite kernel, the monodromy action of $\Mon(X)$ also gives rise to an action of $G_X$ on $H^\ast(X)$ of $X$ via automorphisms (cf.~\cite[Proposition 4.1]{HHT12}). 
  
The $G_X$ action also respects the Hodge structure. For instance,  the trivial summands in $H^\ast(X,\QQ)$ are  Hodge classes. In particular, the Chern classes $c_i(T_X)$ of the tangent bundle ly in a trivial summand of $H^{2i}(X,\QQ)$ and $c_i(T_X) =0$ is zero when $i$ is odd.

\begin{remark}
The theory of Lefschetz modules, developed by Verbitsky \cite{Ver95} and Looijenga-Lunts \cite{LL}, provides another  action of the group $$\widetilde{G}_X:=\mathrm{Spin} (H^2 (X , \QQ ))$$ on  $H^\ast(X,\QQ)$. The action also preserves the ring structure and on the even cohomology it factors through a representation 
 \begin{equation} \label{rho}
  \rho : G_X\to \mathrm{Aut} (H^{\rm even} (X , \QQ ))
 \end{equation}
 that coincides with the standard representation on $H^2(X,\QQ)$. Therefore, the two representation of $\Mon^2(X)$ in $\GL(H^\ast(X,\QQ))$ agree on  a finite index subgroup (cf.~\cite[\S 4.6]{Mar08}). 
 \end{remark}

%%%%%%% (Can be added later)  Furthermore, if we only consider the cohomology classes generated by $H^2(X,\QQ)$, we obtain the following result due to Verbitsky \cite{Ver95}. \begin{proposition}Let $\bar{H}^\ast (X,\QQ)\subseteq H^\ast(X,\QQ)$be the subring generated by the second cohomology via cup product. Then\begin{equation}\bar{H}^{2i}(X,\QQ)\cong  \begin{cases} \mathrm{Sym}^{i}H^2(X,\QQ) & \mbox{if } i< n\\ \mathrm{Sym}^{2n-i}H^2(X,\QQ) & \mbox{if } i\geq n.\end{cases}\end{equation}	\end{proposition}

\subsection{Beauville-Voisin classes} 
We now recall Beauville and Voisin's work on {\it weak splitting property} of algebraic hyperk\"ahler manifolds, which is observed from the following result on the Chow ring of K3 surfaces. 

\begin{theorem}[Beauville-Voisin \cite{BV04}]
	Let $X$ be a smooth projective K3 surface. Then there exists a canonical zero cycle $c_X\in \CH_0(X)$ of degree one such that
\begin{itemize}
\item $c_X$ is represented by a point on a rational curve on $X$
\item  for any divisors $D_1,D_2\in \Pic(X)$, the intersection $D_1\cdot D_2$ is proportional to $c_X$ in $\CH_0(X)$. 
\item $c_2(T_X)=24c_X$. 
\end{itemize}
 In particular, the natural map 
\begin{equation}\label{wsp}
 \mathrm{DCH}^\ast(X)\rightarrow H^\ast(X,\QQ)
\end{equation}
is  injective. 
\end{theorem}

Beauville has pointed out that this can be understood in a broader framework of Bloch-Beilinson-Murre filtration, the so called {\it weak splitting property}, and he conjectures that this can be generalized to hyperk\"ahler manifolds, i.e. \eqref{wsp} holds when $X$ is hyperk\"ahler. 
This was later strengthen by Voisin \cite{Vo12} via involving the Chern classes of the tangent bundle $T_X$, i.e. replacing $\mathrm{DCH}^\ast(X)$ in \eqref{wsp} by the subalgebra $\BV^\ast(X)\subseteq \CH^\ast(X)$  generated by  divisor classes and  $c_i(T_X)$.  We will call the elements in $\BV^\ast(X)$ the {\it  Beauville-Voisin (BV) classes } of $X$.
 
The conjecture has been confirmed by Voisin \cite{Vo08} and Fu \cite{Fu15} when $X$ is a K3 surface, Fano variety of lines of a cubic fourfold,  generalized Kummer varieties and $S^{[n]}$ if $n\leq 2b_{2,tr}(S)+4$, where $b_{2,tr}(S)$ is the second Betti number of $S$ minus its Picard number. In these cases, there exists a canonical zero cycle $c_X\in \CH_0(X)$ such that $\BV^n(X)$ is spanned by $c_X$. In general, the existence of $c_X$ remains widely open.  Note that if the canonical cycle $c_X$ exists, then the top Chern class $c_{2n}(T_X)$ is a multiple of $c_X$.

\section{Moduli space of hyperk\"ahler manifolds}

\subsection{Moduli of polarized hyperk\"ahler manifolds} In the sequel we will always assume that a $2n$-dimensional hyperk\"ahler manifold $X$ of type $\Lambda$ has
a primitive polarization $H$, i.e. $c_1(H)\in H^2(X,\ZZ)$ is primitive.  In order to discuss the moduli space of polarized hyperk\"ahler manifolds,  we shall choose a polarization type, i.e. an $\mathrm{O}(\Lambda)$-orbit of a primitive vector $h\in\Lambda$. We say that $(X,H)$ is {\it $h$-polarized} if  $\varphi(c_1(H))=h$ for a given  marking $\varphi:H^2(X,\ZZ)\cong \Lambda$.

Consider the moduli stack of $h$ primitively polarized hyperk\"ahler manifolds of dimension $2n$ and type $\Lambda$ with a given Fujiki invariant, which associates to a scheme $T$ (over $\CC$)  the set of isomorphism classes of flat families of primitively polarized hyperk\"ahler manifolds over $T$.
This stack in  general is  disconnected with only finitely many   connected components.  Let $\cF_h$  be a connected component of this stack.  Then a standard result is  
\begin{proposition}
 $\cF_h$ is a smooth Deligne-Mumford quotient stack and it can be coarsely represented by a quasi-projective variety $F_h$. 
\end{proposition}
\begin{proof}
The proof is the same as in the case of K3 surfaces. The quasi-projectivity of $F_h$ follows from the Torelli theorem and the standard result of Baily-Borel \cite{BB66}. 
\end{proof}

\begin{remark}
	Here we work over the base field $\CC$ only for simplicity. We also refer to \cite{And96} for the results of moduli space of hyperk\"ahler varieties over fields of characteristic $0$. 
\end{remark}

\subsection{Period map} Let us recall Verbitsky and Markman's Hodge theoretic Torelli theorem for hyperk\"ahler manifolds.
Given a polarization type vector $h\in\Lambda$, the orthogonal completement $\Lambda_h:=h^\perp\subseteq \Lambda$ is a lattice  of signature $(2,b_2(X)-3)$. The period domain $D$ of $h$-polarized hyperk\"ahler manifolds of type $\Lambda$ is a connected component of  
\begin{equation}\label{domain}
D^\pm=\{ x\in \PP(\Lambda_h\otimes \CC)~|~(x,x)=0, (x,\bar{x})>0\},
\end{equation}
seen as a Hermitian symmetric domain of type IV.
Consider the connected component of the coarse moduli space of marked hyperk\"ahler manifolds $$N_h= {(X, H,\varphi)}/\sim$$
parametrizing  marked $h$-polarized hyperk\"ahler manifolds $(X, H,\varphi)$ up to equivalence. Then we have a period map
\begin{equation}\label{period}
\widetilde{\cP}_h: N_h\rightarrow D
\end{equation}
that maps $(X,H,\varphi)$ to the line $[\varphi(\omega_X)]$.

\begin{remark}
One can also consider the moduli (analytic) stack $\widetilde{\cN}_h$ of fully marked $h$-polarized hyperk\"ahler manifolds and denote by $\widetilde{N}_h$ its coarse moduli space.  This also defines a period map $\widetilde{N}_h\rightarrow D$. See \S \ref{Tei} for more details. 
\end{remark}

Write $\Gamma_h:=\Mon^2(X,H)$, then Markman shows that the map \eqref{period} is a $\Gamma_h$-equivariant open immersion (cf.~\cite{Mar11}). To forget the marking, we quotient \eqref{period} on both sides and get

\begin{theorem}\cite[Theorem 8.4]{Mar11}
The period map \eqref{period} induces  an open immersion
\begin{equation}\label{torelli}
\cP_h:F_h\hookrightarrow  \Gamma_h\backslash D.
\end{equation}  
\end{theorem}
The analytic orbifold $\Gamma_h\backslash D$ has an algebraic structure and it is a quasiprojective variety.
In this paper, we may also use $\cP_h$ to denote the map from the stack $\cF_h$ to $\Gamma_h\backslash D$ by abuse of notations.

\begin{remark}
If we consider the coarse moduli space of quasipolarized hyperk\"ahler manifolds, i.e. the line bundle $H$ is big and nef,  the open immersion \eqref{torelli} becomes an isomorphism, see \cite{Huybrechts99}.  
\end{remark}

\subsection{~}\label{Tei} There is an analytic construction of $\cF_h$ via Teichm\"uller theory. As the hyperk\"ahler manifolds only differ by their complex structures, we can write $X$ as $(M,I)$,  where  $M$ is a hyperk\"ahler manifold and  $I$ is a complex structure on $M$.  Then the moduli space $\cF_h$ can be described as  the orbifold $$\mathrm{Tei}_h(M)/\mathrm{MCG}(M,h),$$ given as the quotient of a divisor $\mathrm{Tei}_h(M)$ in a connected component of the  Teichm\"uller space  of $M$ by a subgroup $\mathrm{MCG}(M,h)$ of the mapping class group. 

The space $\mathrm{Tei}_h(M)$  actually coarsely represents the  moduli stack $\widetilde{\cN}_h$ in Remark 3.2.1 and  the image of the  $\rm{MCG}(M,h)$ in $\GL(H^\ast(M,\ZZ))$ is nothing but the polarized monodromy group $\Mon(X,H)$. In this way, most of the results we discussed in this section can be found in \cite{Ve13}  using the language of Teichm\"uller spaces.

\subsection{Level structures}\label{Slevel}  We now introduce the level structures on polarized hyperk\"ahler manifolds to help rigidify our moduli problem.  Just as the case for moduli of curves with level structures, this can be viewed as the quotient of $\mathrm{Tei_h(M)}$ by a finite index subgroup of $\mathrm{MCG}(M,h)$. For polarized hyperk\"ahler manifolds, there are two natural ways to add level structures by taking the kernel of either the map 
\begin{equation}\label{totallevel}
\mathrm{MCG}(M,h)\rightarrow \GL(H^\ast_{\rm{prim}}(M,\ZZ/\ell\ZZ))
\end{equation}
 or the  map
 \begin{equation}
\mathrm{MCG}(M,h) \rightarrow \mathrm{O}(H^2_{\rm{prim}}(M,\ZZ/\ell\ZZ)),
 \end{equation} 
 for $\ell>0$.  In this paper, we mainly concern about the moduli space of polarized hyperk\"ahler manifolds with a level strucutre on the total cohomology, i.e. the quotient of $\mathrm{Tei}_h(M)$ by the kernel of \eqref{totallevel}. 
 
 In the algebraic setting, we say that a  {\it full $\ell$-level structure} on a $h$-polarized hyperk\"ahler $(X,H)$  is an isomorphism $$H^\ast(X,\ZZ/\ell\ZZ)\cong H^\ast(M,\ZZ)\otimes \ZZ/\ell\ZZ,$$
 mapping  the class $c_1(H)$ to $h$. 
  We let $\cF^\ell_h$ be the connected component (associated to $\cF_h$) of the moduli stack of $h$-polarized hyperk\"ahler manifold with a full $\ell$-level structure.  The forgetful map 
$\cF_{h}^\ell\rightarrow\cF_h $
 is finite and \'etale.  
 
  Denote by $$\Mon_\ell (X,H)\subseteq\Mon(X,H)$$  the  polarized monodromy group with a full $\ell$-level structure, which is the image of the kernel of \eqref{totallevel} in $\Mon(X,H)$. If $\ell$ is sufficiently large, the projection to the second cohomology yields an isomorphism
  \begin{equation}
  \Mon_\ell (X,H)\cong \Mon^2_{\ell}(X,H)
  \end{equation} 
 because of \eqref{monproj}. We  let $\Gamma^\ell_h\subseteq \Gamma_h$ be the   image of  $\Mon^2_\ell(X,H)$ via  markings.  Then the coarse moduli space $F_h^\ell$ of $\cF_h^\ell$ is the quotient of  the coarse moduli space of fully marked $h$-polarized hyperk\"ahler manifolds
  by $\Mon^2_\ell(X,H)$ and it admits an open immersion 
 \begin{equation}\label{level-period}
 \cP_\ell: F_h^\ell\rightarrow \Gamma_h^\ell \backslash D.
 \end{equation}
 induced by the period map.

If the kernel of \eqref{aut}  is trivial,  then   $\cF_{h}^\ell$ is represented by $F_{h}^\ell$ for $\ell \gg 1$  because any object in $\cF_{h}^\ell$ has only trivial automorphism  (cf.~\cite[Lemma 1.5.12]{Ri06}). In this case,  $F_{h}^\ell$ carries a universal family in the category of varieties. As mentioned in \S 2.2,  this has been confirmed for all known examples except {\bf OG6}.

\subsection{~}Let us give  a few remarks about the second construction.  Rizov \cite{Ri06} has dealt with the case of  K3 surfaces in adelic language (see also \cite{Ma12,Pe15}), but there is no difficulty to  extend his construction to  hyperk\"ahler manifolds, see e.g.~\cite[\S 3.3]{Ch13}. The resulting moduli space better fits with our notions of Shimura varities defined in Section~\ref{Sec5}.

Write $G=\SO(\Lambda_h)$ and let $K\subseteq G(\AA_f)$ be an open compact subgroup. Recall that there is an injective morphism 
$$\{g\in \SO(\Lambda) \; | \; g(h)=h\}\rightarrow \SO(\Lambda_h).$$ Following \cite[2.2]{Ri06}, we say that $K$ is {\it admissible}  if every element of $K$ can be viewed as an isometry of $\Lambda(\AA_f)$ fixing $h$ and  stabilizing $\Lambda_h(\widehat{\ZZ})$. In this case, we define
a {\it $K$-level structure}  on a $h$ primitively polarized hyperk\"ahler $(X,H)$  as an element of the set 	
\begin{equation}
\left\{ K\backslash \{g\in {\rm Isometry} (\Lambda(\hat{\ZZ}), H^2(X,\hat{\ZZ})(1)) \; | \; g(h)=c_1(H) \} \right\}.
\end{equation}
	
As in \cite[Definition 1.5.16]{Ri06}, one can  consider the moduli stack $\cF_{h,K}$ of primitively $h$-polarized hyperk\"ahler manifolds with a $K$-level structure and  let $F_{h,K}$ be the corresponding coarse moduli space. The period map gives an open immersion 
\begin{equation}
F_{h,K}\hookrightarrow \Gamma_{h,K}\backslash D,
\end{equation}
where $\Gamma_{h,K}=\Gamma_h\cap K$ (cf.~\cite[Proposition 3.2.1]{Ri06}). Again, this  follows from Verbitsky-Markman's Hodge theoretic Torelli theorem on hyperk\"ahler manifolds.

\subsection{Descent of local systems} \label{SDescend}
Let $\pi:\cU_{h}\rightarrow \cF_{h}$ be the universal family of $h$-polarized hyperk\"ahler manifolds over $\cF_{h}$.   We denote by  $$\HH_\pi:=R \pi_\ast \QQ$$  the local system   that fiberwise corresponds to the cohomology algebra  of the fiber.  As we explained in \S 2.3, the speical orthogonal group $\SO(\Lambda_h)$ (up to a finite covering) acts fiberwisely on $\HH_\pi$.  As $\SO(\Lambda_h)$ acts naturally on $\Gamma_h\backslash D$, our goal is to show that the local system $\HH_\pi$ descends to an automorphic local system on $\Gamma_{h}\backslash D$  after taking a finite cover of $\cF_{h}$. We prove it by adding level structures defined in \S \ref{Slevel}. 
\begin{lemma}\label{Descend}
Consider the universal family
\begin{equation}
\pi_\ell:\cU^\ell_h \rightarrow \cF_{h}^\ell
\end{equation}
of h-polarized hyperk\"ahler manifolds with a full $\ell$-level strucutre. When $\ell$ is sufficiently large,  the local system  $\HH_{\pi_\ell}:=R (\pi_\ell)_\ast \QQ $ descends to an automorphic local system $\bH^\bullet$ on $\Gamma_h^{\ell}\backslash D$ via the period map  \eqref{level-period}. 
\end{lemma}
\begin{proof} 
	 We can consider the universal fibration $\widetilde{\pi}:U\rightarrow \widetilde{\cN}_h$ over $\widetilde{\cN}_h$. Let $\HH_{\widetilde{\pi}}$ be the associated local system. One can easily see that the local system $\HH_{\widetilde{\pi}}$ is a constant system on $\widetilde{\cN}_h$ and it descends to the constant system $$D\times H^\ast(M,\QQ)$$ on $D$ via the period map \eqref{period} (cf.~\cite{Mar17}). 
	 
For $\ell$ sufficiently large, the action of  $\Gamma_h^\ell\cong \Mon_\ell(X,H)$  on $\widetilde{\cN}_h$ naturally lifts to an action on $\HH_{\widetilde{\pi}}$. 
After taking the quotient by $\Gamma_h^\ell$, it gives the locally constant system $\HH_\pi$ on $\cF^\ell_h$. On the other hand, the map \eqref{level-period} is obtained by passing to the quotient on both sides of \eqref{period}. It follows that  $\HH_\pi$ descends to the automorphic bundle 
$\bH^\bullet:=\Gamma_\ell\backslash (D\times H^\ast(M,\QQ))$ via \eqref{level-period}.
\end{proof}

%\begin{remark}If we regard $F_h$ as an open subset of $D$, there is  an isomorphism \begin{equation}	H^{i}(\cF_h,\HH)\cong H^i(F_h, \HH). 	\end{equation}	via the coarse moduli map. \end{remark}

%\begin{remark} \label{RemakDescend}If we consider the local subsystem $\bar{\HH}_\pi\subseteq \HH_\pi $ that fiberwise corresponds to the subring $\bar{H}^* (-)$ generated by the second cohomology  of the fiber, then  Lemma \ref{Descend} holds for $\bar{\HH}_\pi$ by the same argument.	\end{remark}

\subsection{Moduli of lattice polarized hyperk\"ahler} \label{Slattice}
 We have to introduce the lattice polarized hyperk\"ahler manifolds for later use. The moduli space of lattice polarized hyperk\"ahler manifolds are briefly studied by Camere in \cite{Cam16, Cam15}, which generalize the work of Dolgachev \cite{Dol96} on K3 surfaces. Here we review the basic notions with certain modifications. Let $$j:\Sigma\hookrightarrow \Lambda$$ be a fixed primitive embedding of a lattice $\Sigma$ with  signature $(1, r)$.  An  ample {\it $\Sigma$-polarized hyperk\"ahler  manifold} (with respect to $j$) is a pair $(X,\phi)$,  where $X$ is a hyperk\"ahler manifold and 
$$\phi :  \Sigma\rightarrow \mathrm{NS}(X) \subset H^2 (X , \mathbb{Z})$$
is a primitive lattice embedding satisfying 
\begin{itemize}
	\item  $\phi(\Sigma)$  contains an ample divisor class of $X$,
	\item there exists a marking $\varphi$ such that $\varphi\circ\phi=j$. 
\end{itemize}

If we fix an ample class $h$,  we can define the so called {\it  $h$-ample $\Sigma$-polarized hyperkahler manifold (with respect to $j$)} as a triple $(X,H,\phi)$ satisfying that $(X,H)$ is $h$-polarized and  $\phi(\Sigma)$ contains $c_1(H)$. Denote by $\cF_{\Sigma,h}$ a connected component of the moduli stack of  $h$-ample $\Sigma$-polarized hyperk\"ahler manifolds. It comes with the natural forgetful map $\cF_{\Sigma,h}\rightarrow \cF_h$. 

There is also a Hodge theoretic description of $\cF_{\Sigma,h}$ by Torelli theorem. Regard $\Sigma$ as a sublattice of $\Lambda$ and we consider the restricted period domain  $D_\Sigma\subseteq D$ as a connected component of 
  $$\{ x\in \PP(\Sigma^\perp \otimes \CC)~|~(x,x)=0, (x,\bar{x})>0\}.$$

Let $\Gamma_\Sigma$ be  the stabilizer of $\phi(\Sigma)$ in the monodromy group ${\rm Mon}^2(X)$. It is an arithmetic subgroup of $\mathrm{O}(\Sigma^\perp)$.  As before, let $F_{\Sigma,h}$ be the corresponding coarse moduli space. Then the restriction of  period map on moduli space of marked polarized hyperk\"ahler manifolds defines a map
$$  \cP_{\Sigma}:F_{\Sigma,h} \rightarrow \Gamma_\Sigma\backslash D_{\Sigma}.$$
which is an open immersion and there is a commutative diagram 
\begin{equation}
\xymatrix{ F_{\Sigma,h} \ar[r]^{\cP_\Sigma}\ar[d]^{} &\ \Gamma_\Sigma\backslash D_\Sigma\ar[d]^{}\\ F_h \ar[r]^{\cP_h} & \Gamma_h\backslash D .}
\end{equation}  
  
 \begin{remark}
 	In general, the definition of $\cF_{\Sigma,h}$  depends on the primitive embedding of $\Lambda$. In some cases, such as $\Lambda$ is unimodular,  the primitive embedding of $\Sigma$ to $\Lambda$ is unique up to isometry and $\cF_{\Sigma,h}$ is thus independent of the choice of $j$.   
 \end{remark}

Finally, one can easily see that  the results in \S\ref{Slevel} and \S\ref{SDescend} naturally extend to lattice polarized hyperk\"ahler manifolds. More precisely, concerning the full $\ell$-level structures on  lattice polarized hyperk\"ahler manifolds, we let $\cF_{\Sigma,h}^\ell$ (resp.~$F_{\Sigma,h}^\ell$) be the (coarse) moduli space of $h$-ample $\Sigma$-polarized hyperk\"ahler manifolds with a full $\ell$-level structure. Then there is an open immersion
\begin{equation}
\cP^\ell_\Sigma:F^\ell_{\Sigma,h}\hookrightarrow  \Gamma_\Sigma^\ell \backslash D_\Sigma,
\end{equation}
for some finite index subgroup $\Gamma_\Sigma^\ell\subseteq \Gamma_\Sigma$ and the associated local systems on $F^\ell_{\Sigma,h}$ descend via $\cP^\ell_\Sigma$.

%  See \cite{Cam16} for more details. Similarly, one can see that  the results  in $\S$3.3 and $\S$3.4 also holds for  $\cF_{\Sigma,h}$. That is, the moduli space of $\Sigma$-polarized hyperk\"ahler manifolds with a level structure admits a period map to $\Gamma_{K,\Sigma}\backslash D_\Sigma $ for some finite index subgroup $\Gamma_{K,\Sigma}\subseteq \Gamma_\Sigma$ and Lemma \ref{Descend}, Lemma \ref{Descend2} hold for $\cF_{\Sigma,K}$.  

\section{Cycle classes on moduli of hyperk\"ahler manifolds}

Analogous to the case of K3 surfaces, there  exist natural  cycle classes on $\cF_h$ and its universal family, which comes from either Hodge theory or the geometry of the generic fiber. In this section, we discuss the relations between these cycles and define the tautological ring on the moduli space. Natural conjectures are made motived from MOP conjecture and generalized Franchetta conjecture. 

\subsection{Noether-Lefschetz cycles}
   We first recall the Noether-Lefschetz theory for hyperk\"ahler manifolds. Let $\pi:\cU\rightarrow \cF$ be a family of projective hyperk\"ahler manifolds over $\cF$. When $\cF$ is  a quasiprojective variety, we define the $r$-th  {\it Noether-Lefschetz loci on $\cF$} (with respect to $\pi$)  as 
\begin{equation}
\cN^r(\cF):=\{ b\in \cF~|~\rank (\Pic(\cU_b))\geq r+\rho\}.
\end{equation}
where $\rho$ is the Picard number of the generic fiber of $\pi$.  We can also define the Noether-Lefschetz loci with respect to  the universal family $\pi:\cU_{h}\rightarrow \cF_{h}$ over the moduli stack $\cF_{h}$,  where the irreducible component of the Noether-Lefschetz loci on $\cF_{h}$ will occur as the proper map
\begin{equation}
\cF_{\Sigma,h}\rightarrow \cF_{h},
\end{equation} 
for some Lorentian lattice $\Sigma\hookrightarrow \Lambda$ containing $h$. In either cases, we define 
\begin{equation}
\NL^\ast_\pi(\cF)\subseteq \CH^\ast(\cF)
\end{equation}
as the subalgebra generated by Noether-Lefschetz loci on $\cF$ (cf.~\cite{PY16}). We may omit the subscript $\pi$ when $\cF$ is the moduli stack. 

%For the universal family $\cU_h\rightarrow \cF_h$ of $h$-polarized hyperk\"ahler manifolds, the locus  $\cN^r(\cF_{h} )$ is a union of countably many irreducible substacks and each of them is of the form $$\cN^r(\cF_{h}^\circ )\rightarrow \cF_h$$ for some lattice $\Sigma$ of rank $r$. We call the corresponding class $[\cF_{\Sigma,h}]\in \CH^r(\cF_h)$ an {\it irreducible NL-cycle of codimension $r$} and define the {\it Noether-Lefschetz ring}\begin{equation}\NL^\ast (\cF_h) \subseteq\CH^\ast(\cF_h)\end{equation} to be the $\QQ$-subalgebra generated by  irreducible NL-cycles $[ \cF_{\Sigma,h}]$ on $\cF_h$ for all possible $\Sigma$. Similarly, we may also define the irreducible NL-cycles of codimension $i$ on $\cF_{\Sigma,h}$ by taking the irreducible component of the locus $\cN^i(\cF_{\Sigma,h})$.  An important fact is the following

\begin{proposition}[see~\cite{BLMM16}] \label{NLconj} The first Chern class of the Hodge line bundle is in the span of NL-divisors. Moreover,  let $m_\Lambda=\rank ~\Lambda$ be the second Betti number of hyperk\"ahler manifolds in $\cF_{\Sigma,h}$ and assume that $m_\Lambda-\rank (\Sigma)\geq 5$, then $$\NL^1(\cF_{\Sigma,h})=\Pic (\cF_{\Sigma,h}).$$
In particular, we have an inclusion $\mathrm{DCH}^\ast(\cF_{\Sigma,h})\subseteq\NL^\ast(\cF_{\Sigma,h})$.
\end{proposition}	
Let us roughly explain the proof. It actually relies on the study of Picard groups on the coarse moduli space $F_{\Sigma,h}$ as we have $\Pic(\cF_{\Sigma,h})\cong \Pic(F_{\Sigma,h})$.   
In  \cite[Theorem 3.8]{BLMM16},  we deal with orthogonal Shimura varieties whereas $F_{\Sigma,h}$ is only an open subset of $\Gamma_\Sigma \backslash D_\Sigma$. However the map 
$$\cP_{\Sigma}^* : \Pic (\Gamma_\Sigma \backslash D_\Sigma ) \to \Pic (F_{\Sigma,h})$$ 
is onto and we are easily reduced to $\Gamma_\Sigma \backslash D_\Sigma$, see Proposition \ref{P423} below for details. Then we can conclude the assertion by applying \cite[Theorem 3.8]{BLMM16}.  Here, the assumption $m_\Lambda-\rank (\Sigma)\geq 5$ or equivalently $\dim D_\Sigma \geq 3$ is a necessary condition  for us to apply \cite[Theorem 3.8]{BLMM16}.

%The latter implies that the subspace $H^{1,1} (\widetilde{\Gamma}_\Sigma \backslash D_\Sigma) \subset H^2 (\widetilde{\Gamma}_\Sigma \backslash D_\Sigma , \mathbb{C})$ is defined over $\mathbb{Q}$ and spanned by cycle classes of  Noether-Lefschetz divisors. Finally the first Chern class map $$c_1 : \Pic (\widetilde{\Gamma}_\Sigma \backslash D_\Sigma) \to H^2 (\widetilde{\Gamma}_\Sigma \backslash D_\Sigma , \QQ)$$ is an injection since $H^1 (\widetilde{\Gamma}_\Sigma \backslash D_\Sigma , \QQ) = 0$, see \cite[Remark 4.2]{BLMM16}. We conclude that $$\Pic (\widetilde{\Gamma}_\Sigma \backslash D_\Sigma) = H^{1,1} (\widetilde{\Gamma}_\Sigma \backslash D_\Sigma) \subset H^2 (\widetilde{\Gamma}_\Sigma \backslash D_\Sigma , \mathbb{C})$$ is spanned by cycle classes of  Noether-Lefschetz divisors. 

\subsection{Kappa classes and tautological ring}\label{3.2} 
In this subsection, we construct the $\kappa$-classes on $\cF_{\Sigma,h}$ by taking the pushforward of BV classes on the generic fiber of  universal families of hyperk\"ahler manifolds of (complex) dimension $n$. 
This idea is initiated  in \cite{MOP15} for the case of K3 surfaces. To work in greater generality, we first define $\kappa$-classes with respect to an arbitrary smooth family.

\begin{definition} \label{D321}
Let $\pi: \cU\rightarrow \cF$ be a smooth projective family of hyperk\"ahler manifolds over a smooth  Deligne-Mumford stack.  
Let $$\cB=\{\cL_0,\ldots, \cL_r\}\subseteq \Pic (\cU)$$ be a collection of line bundles whose image in $\Pic (\cU/\cF)$ form a basis. Then we define the {\it $\kappa$-classes}
	\begin{equation}
	\kappa_{a_0,\ldots, a_r;b_1,\ldots, b_{2n} }^\cB=\pi_\ast(\prod\limits_{i=0}^r c_1(\cL_i)^{a_i} \prod\limits_{j=1}^{2n} c_{j}(T_{\pi})^{b_j} )\in \CH^{m}(\cF),
	\end{equation}	
	where $m=\sum\limits_{i=0}^r a_i+\sum\limits_{j=1}^{2n} jb_j - 2n$.   For simplicity, we may write  
	\begin{equation}
	\kappa^\cB_{a_0,\ldots, a_r}=\kappa^\cB_{a_0,\ldots, a_r;0,\ldots,0}.
		\end{equation}
	and refer to these as the {\it special $\kappa$-classes}. %These are obtained by pushing forward the $\mathrm{DCH}^\ast$-classes of the generic fiber. Nicolas: I commented this out because we have to vary $\cB$ to have this equality no ?
\end{definition}

In particular, we can define $\kappa$-classes and also special $\kappa$-classes on $\cF_{\Sigma,h}$ as the $\kappa$-classes with respect to the universal family  $\pi_\Sigma:\cU_\Sigma\rightarrow \cF_{\Sigma,h}$.  This yields

 \begin{definition}
 The {\it  tautological ring}  $\mathrm{R}^\ast(\cF_{h})\subseteq \CH^\ast(\cF_{h})$  of $\cF_h$ is defined to be the subalgebra generated by the images of all $\kappa$-classes on $\cF_{\Sigma,h}$ via the pushforward maps
 	 \begin{equation}\label{pushforward}
 	 (i_\Sigma)_\ast: \CH_k(\cF_{\Sigma,h} )\rightarrow \CH_k(\cF_h)
 	 \end{equation}
 	   for all $\Sigma$. Instead, if we only use the {\it special} $\kappa$-classes on $\cF_{\Sigma,h}$ via \eqref{pushforward}, we obtain a subring $\mathrm{DR}^\ast(\cF_h)\subseteq \mathrm{R}^\ast(\cF_h)$. \end{definition}

Clearly, we have natural  inclusions
 \begin{equation}
  \NL^\ast(\cF_h)\subseteq \mathrm{DR}^\ast(\cF_h) \subseteq \mathrm{R}^\ast(\cF_h)
  \end{equation} 
from the definition. Note that we can vary the choices $\cB$ of universal line bundles on $\cU_\Sigma$. (These choices differ by the pullback of line bundles on $\cF_{\Sigma,h}$.) It follows that $\mathrm{DR}^\ast(\cF_h)$ and hence $\mathrm{R}^\ast(\cF_h)$ contains all  classes in $\mathrm{DCH}^\ast(\cF_{\Sigma,h})$ via \eqref{pushforward}. Then we propose the following generalization of MOP conjecture: 
\begin{conjecture}[Generalized Tautological Conjecture]\label{GTC}
	$$\NL^\ast(\cF_h)=\mathrm{R}^\ast(\cF_{h}).$$
\end{conjecture}

\subsection{Cohomological tautological conjecture}As a weaker version of the generalized tautological conjecture, we mainly consider the cohomological tautological conjecture as follows.
Let $H^\ast(\cF_h,\QQ)$ denote the singular cohomology  of $\cF_h$ which is isomorphic to the singular cohomology of the coarse moduli space $F_h$ (cf.~\cite{Be04}). There is a cycle class map 
\begin{equation}
cl: \CH^\ast(\cF_h)\rightarrow H^\ast(\cF_h,\QQ)，
\end{equation}
and we add subscript hom to denote the image of the corresponding ring in $H^\ast(\cF_h,\QQ)$ via the cycle class map. Then we have the cohomological tautological conjecture (CTC).

\begin{conjecture}\label{conj5}
	$\NL^\ast_{\rm hom}(\cF_h)=\mathrm{R}^\ast_{\rm hom}(\cF_{h}).$
\end{conjecture}

One of our main results is 
\begin{theorem}\label{main2} Assume that the lattice $\Lambda$ has rank $m_\Lambda\geq 6$, or equivalently $\dim \cF_h\geq 3$. Then we have 
\begin{equation}
\NL^\ast_{\rm hom}(\cF_h)=\mathrm{DR}^\ast_{\rm hom}(\cF_h).
\end{equation}
Moreover,  if $n< \frac{m_\Lambda-3}{8}$,  then
\begin{equation}
\NL^\ast_{\rm hom}(\cF_h)=\mathrm{R}^\ast_{\rm hom}(\cF_h).
\end{equation} In particular,  Conjecture \ref{conj5} holds when $\Lambda$ is of $\mathrm{K3}^{[n]}$-type with $n\leq 2$.

\end{theorem}

We prove Theorem \ref{main2} in Section 8. Indeed, we will prove a more general result concerning $\kappa$-classes with respect to general  families of hyperk\"ahler manifolds. 
Unfortunately, our results can not imply the full Conjecture 5 for hyperk\"ahler manifolds of  generalized Kummer type, {\bf OG6}, {\bf OG10}, or K3$^{[n]}$-type with large $n$ as their second Betti number is relative small. Instead, we show that a large part of the tautological classes are lying in R$^\ast_{\hom}(\cF_h)$.

\subsection{Examples}\label{GRR} As a completement to Theorem 4.3.1,  we follow \cite{GT05} to  give some examples of $\kappa$-classes on $\cF_h$ with  $a_i=0$ and $b_j\neq 0$ lying in $\NL^\ast(\cF_h)$. For simplicity,  we may write $$\tilde{\kappa}_{b_1,\ldots,b_{2n}}=\kappa_{0,\ldots,0; b_1,\ldots;b_{2n}}. $$ 
Let $\Omega_{\cU_h/\cF_h}^i$ be the relative sheaf of differential $i$-forms and  we denote by $\lambda=c_1(R^0\pi_\ast \Omega^2_{\cU_h/\cF_h})$ the first Chern class of the line bundle  on  $\cF_h$ that fiberwise corresponds to the vector space spanned by the  non-degenerate holomorphic 2-form of the fiber. By \cite[Corollary 8.5]{BLMM16}, we know that $\lambda\in \NL^\ast(\cF_h)$. Easily one can get    $$c_1(R^0\pi_\ast\Omega^{2i}_{\cU_h/\cF_h})=i\lambda\in\NL^\ast(\cF_h). $$
Furthermore, note that there is an isomorphism $ \wedge^i(R^2\pi_\ast \cO_{\cU_h})\cong R^{2i}\pi_\ast \cO_{\cU_h}$ via cup product and $R^{2n}\pi_\ast \cO_{\cU_h}\cong (R^0\pi_\ast\Omega_\pi)^\vee$, this yields  $$c_1(R^{2i}\pi_\ast \cO_{\cU_h})=-i \lambda.$$ 

Using  Grothendieck-Riemann-Roch Theorem, we have 
\begin{equation}\label{GRR}
\begin{aligned}
{\rm ch} (\pi_!\cO_{\cU_h})&= e^{\sum\limits_{i=0}^n  (-i\lambda)}
\\& =  \pi_\ast ({\rm Td}(\Omega^\vee_{\cU_h/\cF_h})).
\end{aligned}
\end{equation} 
Let  $\td_i$ be the $i$-th Todd class of $(\Omega^1_{\cU_h/\cF_h})^\vee=T_{\cU_h/\cF_h}$.  By definition, the class $$\pi_\ast(\td_i)\in \CH^{i-2n}(\cF_h)$$ is  a linear combination of  $\tilde{\kappa}_{b_1\ldots,b_{n}}$ and   we find
\begin{equation}\label{iden}
\pi_\ast(\td_{i+2n})= \sum\limits_{j=0}^n \frac{(-j\lambda)^i}{i!}   
\end{equation}
by comparing terms in \eqref{GRR}  of  degree $i$ at both sides.  The first Chern class $c_1(\Omega^1_{\cU_h/\cF_h})$ is the pull back of a divisor class on $\cF_h$.  By projection formula, this gives 
 \begin{equation}\label{vanish}
 \tilde{\kappa}_{b_1,\ldots,b_{2n}}=0
 \end{equation}
if $\sum\limits_{j\geq 2} jb_j<2n$ or  $\sum\limits_{j\geq 2} jb_j=2n$ with some $b_{2k+1}\neq 0$. When $n=1$, the relations in  \eqref{iden} and \eqref{vanish} suffice to show that all classes $\tilde{\kappa}_{ b_1\ldots,b_{2n}}$ are lying in the subalgebra $\QQ\left<\lambda\right>$ of $\CH^\ast(\cF_h)$ generated by $\lambda$ (cf.~\cite{GT05}). However, it seems  more complicated as $n$ gets larger. 
For instance, let us  compute the first few terms of $\tilde{\kappa}_{b_1,\ldots,b_4}$ (i.e. $n=2$).  Using \eqref{iden}, we have

\begin{itemize}
	\item  $i=0$, $\frac{1}{720}(3\kappa_{0,2,0,0}-\kappa_{0,0,0,1})=3;$
	\item $i=1$, $c_1(\Omega^1_{\cU_h/\cF_h})=2\pi^\ast(\lambda)$; 
	\item  $i=2,~\frac{1}{60480}(\tilde{\kappa}_{0,3,0,0}-9\tilde{\kappa}_{0,1,0,1}-5\tilde{\kappa}_{2,0,0,1}+11\tilde{\kappa}_{2,2,0,0})= \frac{5\lambda^2}{2}$
	 
	 $\vdots$  
\end{itemize}

Let $\chi=\chi(X)$ be the Euler characteristic of  $X\in \cF_h$, e.g. $\chi(S^{[2]})=324$.  It follows that $$\tilde{\kappa}_{0,0,0,1}=\chi, ~\tilde{\kappa}_{0,2,0,0}=720+\frac{\chi}{3},~ \tilde{\kappa}_{2,0,0,1}=4\chi\lambda^2,~\tilde{\kappa}_{2,2,0,0}=(2880+\frac{4\chi}{3})\lambda^2;$$
and 
$$\tilde{\kappa}_{0,3,0,0}-9\tilde{\kappa}_{0,1,0,1}=(\frac{42}{84} +\frac{11\chi}{45360})\lambda^2.$$ 

It would be interesting to know if all such classes are lying in $\QQ\left<\lambda\right>$.

\subsection{Generalized Franchetta conjecture for hyperk\"ahlers}
 Assume that very general members in $\cF_h$ have only the trivial automorphism.  Let $\cF_{h}^\circ\subseteq \cF_{h}$ be the open subset parametrizing the hyperk\"ahler manifolds with only trivial automorphisms and  let $$\pi^\circ:\cU_{h}^\circ \rightarrow \cF_{h}^\circ$$ be the associated universal family. The generalized Franchetta conjecture for $\cU_h^\circ$ is  
 \begin{conjecture}\label{genFranchett}
 Let $\cT_{\pi^\circ}$ be the relative tangent bundle of $\pi^\circ:\cU^\circ_h\rightarrow\cF_h^\circ$. Then for any $\alpha\in \CH^{2n}(\cU_h^\circ)$, there exists $m\in\QQ$ such that $\alpha-mc_{2n}(T_{\pi^\circ})$ supports on a proper subset of $\cF^\circ_h$.
 \end{conjecture}
 
The conjecture can be stated  for the moduli stack and the assumption of trivial automorphism group on the generic fiber of the universal family can be removed.  In \S\ref{cohfran}, we prove the cohomological version of this conjecture for hyperk\"ahler manifolds whose second Betti number is sufficiently large (with respect to the dimension).

\section{Shimura varieties of orthogonal type} \label{Sec5}
In this section, we discuss Kudla-Millson's construction of special cycles on Shimura varieties of orthogonal type as well as its generalization by Funke and Millson to special cycles with coefficients. We also discuss the connection with Noether-Lefschetz cycles on moduli spaces of polarized hyperk\"ahler manifolds.

\subsection{Shimura varieties of orthogonal type} \label{sc}
Let $V$ be a non-degenerate quadratic space  over $\QQ$ of signature $(2,b)$ and let $G=\SO(V)$ the corresponding special orthogonal group. The group $G(\RR )$ of real points of $G$ is isomorphic to $\SO (2,b)$. We denote by $G(\RR)_+ \cong \SO_0 (2,b)$ the component of the identity; recall that it is precisely the kernel of the spinor norm. Let $K_\RR \cong \SO (2) \times \SO (b)$ be a maximal compact subgroup of $G(\RR)_+$ and let $\widehat{D} = G (\RR ) / K_\RR$. We denote by $D$ the connected component of $\widehat{D}$ associated to $G(\RR)_+$. We have $D \cong \SO_0 (2,b) / \SO (2) \times \SO (b)$, it is a Hermitian symmetric domain of dimension $b$.

Let $\widetilde{G}$  be the general spin group $\Gspin(V)$ associated to $V$. For any compact open subgroup $K\subseteq G(\AA_f)$, we set $\widetilde{K}$ to be its preimage in $\widetilde{G}(\AA_f)$ and denote by $X_K$ the double coset
$$\widetilde{G}(\QQ)\backslash \widehat{D} \times \widetilde{G} (\AA_f) /\widetilde{K}.$$
Let $\widetilde{G}(\QQ)_+\subseteq \widetilde{G}(\QQ)$ be the subgroup consisting of elements with totally positive spinor norm, which can be viewed as the subgroup of  $\widetilde{G}(\QQ)$ lying in the identity component of the adjoint group of $\widetilde{G}(\RR )$. Write
$$\widetilde{G}(\AA_f)=\coprod_{j} \widetilde{G}(\QQ)_+ g_j \widetilde{K}.$$ 
The decomposition of $X_K$ into connected components is 
$$X_K=\coprod\limits_{g_j} \Gamma_{g_j}\backslash D,$$
where  $\Gamma_{g_j}$ is the image of $\widetilde{G}(\QQ)_+\cap g_j \tilde{K} g_j^{-1}$ in $\SO_0(2,b)$. When $g_j=1$, we denote by $\Gamma_K$ the arithmetic group $\Gamma_1=K\cap G(\QQ)$ and $Y_K=\Gamma_K\backslash D$ the connected component of $X_K$.

\subsection{Connected cycles with  coefficients}  
Let $U$ be a $\QQ$-subspace of $V$ with $\dim U = r$ and such that the quadratic form is totally negative definite when restricted to $U$. Let 
$\widehat{D}_{U} \subset \widehat{D}$ be the subset consisting of the 2-planes that are perpendicular to $U$ and let $D_U = \widehat{D}_U \cap D$. 

For a fixed compact open subgroup $K\subseteq G(\AA_f)$ and an element $g \in \widetilde{G}(\AA_f)$ we let $\Gamma_{g}$ be the image of $\widetilde{G}(\QQ)_+\cap g \tilde{K} g^{-1}$ in $\SO_0(2,b)$. Denote by $\Gamma_{g, U}$ the image in $\SO_0 (2,b)$ of the pointwise stabilizer of $U$ in $\widetilde{G}(\QQ)_+\cap g \tilde{K} g^{-1}$.
We then have a map 
\begin{equation}\label{cc}
\Gamma_{g, U}\backslash D_U \longrightarrow \Gamma_{g} \backslash D.
\end{equation}

Following Kudla \cite{Ku97} we will denote this connected cycle (with trivial coefficients) by $c(U , g , K)$; it is of codimension $r$. We let $\mathrm{SC}^r (Y_K ) \subseteq \CH^r(Y_K)$ be the subgroup spanned by connected cycles  $c(U, 1. K)$  of codimension $r$ and set 
$$\mathrm{SC}^\ast (Y_K )\subseteq \CH^\ast (Y_K)$$ 
to be the subalgebra generated by connected cycles of all codimensions. We list some simple properties of connected cycles.
\begin{itemize}
	\item  Let $K'\subseteq K$ be an open compact subgroup and let $f:Y_{K'}\rightarrow Y_K$ be the covering map. Then 
	\begin{equation}\label{pullbackcc}
	f(c(U,1,K'))=c(U,1,K)~\hbox{and}~f^\ast(\SC^r(Y_K))\subseteq \SC^r(Y_{K'}).
	\end{equation}
		
    \item  For the map $\iota:\Gamma_{1,U}\backslash D_U\rightarrow Y_K $ in \eqref{cc}, we have \begin{equation}\label{pcc}
    \iota_\ast(\SC_k(\Gamma_{1,U}\backslash D_U))\subseteq \SC_{k}(Y_K).
    \end{equation}
    Here we write $\SC_k(X):=\SC^{\dim X-k}(X)$.

\end{itemize}

% Finally note that any element $\gamma \in \widetilde{G}(\QQ)_+$ acts on $D$ and induces a map $$\Gamma_g \backslash D \to \gamma \Gamma_g \gamma^{-1} \backslash D = \Gamma_{\gamma g} \backslash D.$$Denoting by $\gamma \cdot c(U,g,K)$ the image of the connected cycle $c(U,g,K)$ under this map, we have: $$\gamma \cdot c(U, g , K ) = c(\gamma U , \gamma g , K ).$$

\subsection{~} Now we give some description of the ring $\SC^\ast(Y_K)$. Firstly, besides the connected cycles, there is  Hodge line bundle on $Y_K$ and  we use $\lambda$ to denote its first Chern class in $\CH^1(Y_K)$. It is preserved under the pullback via \eqref{cc} and there is a very useful result
 \begin{lemma}\label{hodge}\cite[Corollary 8.7]{BLMM16} The class 
 	$\lambda\in \SC^1(Y_K).$
 \end{lemma}
  
This yields the following
\begin{theorem}
$\mathrm{SC}^*(Y_K )=\bigoplus\limits_r \mathrm{SC}^r (Y_K ).$
\end{theorem}
\begin{proof}
	We need to show that the  image of the intersection produt 
	$$\SC^{r_1}(Y_K)\times \SC^{r_2}(Y_K)\xrightarrow{\cap} \SC^\ast(Y_K)$$ lies in $\SC^{r_1+r_2}(Y_K)$.
Let us first consider the case  when $r_1=1$.   For any two connected cycles $\alpha=c(U_1,1, K)\in \SC^{r_1}(Y_K)$ and $\beta=c(U_2,1, K)\in \SC^{r_2}(Y_K)$,   there are two possibilities: 
	 
	\begin{enumerate}
		\item   If $\alpha$ and $\beta$ intersect properly, then $\alpha\cdot \beta$ is a linear combination of  connected cycles $c(U,1,K)$ where  $U$ is spanned by $U_2$ and an element in the $\Gamma_K$-orbit of $U_1$.  
		
	\item If $\beta$ is contained in $\alpha$, then 	
		$$\alpha\cdot\beta\in \left<-\lambda\cdot \beta\right>$$
		by O'Grady (cf.~\cite[Lemma 1.2]{Grad86}). By Lemma \ref{hodge} and \eqref{pcc}, we know that $$-\lambda\cdot \beta\in \SC^{1+r_2}(Y_K).$$ It follows that $\alpha\cdot\beta\in \SC^{1+r_2}(Y_K)$.
	\end{enumerate}  
This further gives 
\begin{equation}\label{icc}
\alpha_1\cdot\alpha_2\cdots  \alpha_r \cdot \gamma\in \SC^{r+r_2}(Y_K)
\end{equation}	
	for any $\alpha_1,\ldots,\alpha_r\in \SC^1(Y_K)$ and $\gamma\in \SC^{r_2}(Y_K)$. 

When $r_1>1$,  given a connected cycle $c(U_1,1,K)$ of codimension $r_1$ on $Y_K$,  an easy inductive  argument shows that there exists an open compact subgroup $K'\subseteq K$ such that the connected cycle $$c(U_1,1,K')\subseteq Y_{K'}$$ can be written as an intersection of $r_1$ connected cycles of codimension one, i.e. 
\begin{equation}
c(U_1,1,K')=c(W_1,1,K')\cdot c(W_2,1,K')\cdots c(W_{r_1},1,K'),
\end{equation} 
for some subspace $W_j$ with $\dim W_j=1$. Consider the covering map $f:Y_{K'}\rightarrow Y_K$, we have 
\begin{equation}
c(U_1,1,K')\cdot f^\ast(\beta)\in \SC^{r_1+r_2}(Y_{K'}), \quad \forall \beta\in \SC^{r_2}(Y_K)
\end{equation}
by \eqref{pullbackcc} and \eqref{icc}.  As $f(c(U_1,1,K'))=c(U_1,1,K)$, the projection formula gives $c(U_1,1,K)\cdot \beta\in \SC^{r_1+r_2}(Y_k)$. 

\end{proof}

\begin{remark}
 Similar questions have been asked by Kudla in \cite{Ku97}, where  one  consider the  ring generated by  special cycles instead of connected cycles. We believe our argument might still work. 
\end{remark}

To conclude this paragraph we relate connected cycles to Noether-Lefschetz cycles: we have a period map $\cP_h:\cF_h\rightarrow  \Gamma_h\backslash D$. The following proposition is then straightforward. 

\begin{proposition} \label{P423}
	The irreducible Noether-Lefschetz cycles are the restriction of connected cycles on $\Gamma_h \backslash D$ to  $\cF_h$. In particular, 
	\begin{equation}\label{NLSC}
	\NL^\ast(\cF_h)=\cP^\ast_h ( \mathrm{SC}^\ast(\Gamma_h \backslash D)).
	\end{equation}
\end{proposition}

More generally,  one can easily see that \eqref{NLSC} holds for the Noether-Lefschetz ring on the moduli space of lattice polarized hyperk\"ahler manifolds with level structures. 

\subsection{~}Now following Funke and Millson \cite{FM06} we promote the connected cycles $c(U,g,K)$ to cycles with non-trivial coefficients. First recall that given any partition $\lambda = (\lambda_1 , \lambda_2 , \ldots , \lambda_k)$, with $k = \left[ \frac{b+2}{2} \right]$, the harmonic Schur functor associates to the quadratic space $V$ a finite dimensional $\mathrm{O} (V)$-module $\SS_{[\lambda ]} (V)$ which is irreducible of highest weight $\lambda$. As such it defines a fiber bundle 
\begin{equation} \label{bundle}
\Gamma_g \backslash ( D \times \SS_{[\lambda ]} (V) ) \to \Gamma_g \backslash D.
\end{equation}
We will use $\mathbf{E}$ to denote the associated local system.

Now let $\bx = (x_1 , \ldots , x_r)$ be a rational basis of $U$. The vectors $x_1 , \ldots , x_n$ are all fixed by $\Gamma_{g,U}$. Hence any tensor word in the $x_j$'s is also fixed by $\Gamma_{g,U}$. Given a tableau $T$ on $\lambda$ we denote by $\bx_T $ the corresponding harmonic tensor in $\SS_{[\lambda] } (V)$, see \cite{Fulton}.\footnote{Beware that a tableau is called a semi standard filling in \cite{FM06}.} It is fixed by $\Gamma_{g,U}$ and therefore gives rise to a parallel section of the restriction of $\mathbf{E}$ over the connected cycle $c(U,g,K)$. We define connected cycles with coefficients in $\mathbf{E}$ by setting
$$c(\bx , g , K)_T = c(U , g , K) \otimes \bx_T \mbox{ where } U = \mathrm{span} (x_1 , \ldots , x_n).$$

The scalar product on $V$ induces an inner product on $\SS_{[\lambda] } (V)$. Composing the wedge product with this inner product and integrating the result over $Y_K$ we define a pairing 
\begin{equation} \label{PD}
\langle \ , \ \rangle : H^{2r} (Y_K , \mathbf{E}) \times H_c^{2(b-r)} (Y_K , \mathbf{E}) \to \CC,
\end{equation}
where $H_c^* (- )$ denotes the de Rham cohomology with compact support. It follows from Poincar\'e duality that the pairing \eqref{PD} is perfect. Now if $\eta$ is a compactly supported $\SS_{[\lambda] } (V)$-valued closed $2(b-r)$-form on $Y_K$ we can form the period
$$\int_{c(\bx , g , K)_T} \eta = \int_{c(U , g , K)} (\eta , \bx_T ) .$$
It thus corresponds to the connected cycle with coefficient $c(\bx , g , K)_T$ a linear form on $H_c^{2(b-r)} (Y_K , \mathbf{E})$. Since the pairing \eqref{PD} is perfect this linear form in turn defines a class $[c(\bx , g , K)_T]$ in $H^{2r} (Y_K , \mathbf{E})$. 
We shall refer to the corresponding map as the cycle class map (with coefficients). We denote by $\mathrm{SC}^\ast_{\rm hom}(Y_K,\mathbf{E})$ the subring of $H^\ast (Y_K , \mathbf{E})$ generated by the cycle classes $[c(\bx , g , K)_T]$ (as $\bx$ and $T$ vary).

\begin{remark}
When $\lambda = 0$ the representation $\SS_{[\lambda] } (V)$ is trivial and the cycle class map is obtained by composing the inclusion map $\mathrm{SC}^r (Y_K )\subseteq \CH^r (Y_K)$ with the usual cycle class map $\CH^r (Y_K ) \to H^{2r} (Y_K )$. 
\end{remark}

If $K ' \subseteq K$ is a compact open normal subgroup, the finite covering group $\Gamma_K /\Gamma_{K'}$ that acts on $H^{2r} (Y_{K'} , \mathbf{E})$ preserves the image $\mathrm{SC}^r_{\rm hom}(Y_{K'} , \mathbf{E})$ of the cycle class map and, since the covering projection map $Y_{K'} \to Y_K$ maps connected cycles to connected cycles, we have:
\begin{equation} \label{E45}
\mathrm{SC}^r_{\rm hom}(Y_{K'} , \mathbf{E})^{\Gamma_K /\Gamma_{K'}} = \mathrm{SC}^r_{\rm hom}(Y_{K} , \mathbf{E}).
\end{equation}

\subsection{Zucker's conjecture and Hodge theory}As a key ingredient, 
we shall explain the connection between the ordinary cohomology and $L^2$-cohomology of $Y_K$. As we will see later, the second one is well understood as relative Lie algebra cohomology which naturally links to representation theory.

We assume that $Y_K$ is smooth and fix a local system $\mathbf{E}$ as in the preceding paragraph. 
Let $H^i_{(2)}(Y_K,\bE)$ be the $i$-th $L^2$-cohomology  of $Y_K$ with coefficients in $\bE$. By Hodge theory, the group $H^i_{(2)}(Y_K,\bE)$ is isomorphic to the space of $L^2$-harmonic $i$-forms,  which is a finite dimensional vector space with a natural Hodge structure (see~\cite{BG83}). Let $\overline{Y}_K^{bb}$ be the Baily-Borel-Satake compactification  of $Y_K$. Looijenga \cite{Lo88}, Saper and Stern \cite{SS90} have shown that 
\begin{theorem}[Zucker's conjecture]
	 There is an isomorphism  
	 \begin{equation}\label{Zucker}
	 H^k_{(2)}(Y_K,\bE)\cong IH^k(\overline{Y}_K^{bb},\bE),
	 \end{equation}
	 where $IH^\bullet(\overline{Y}_K^{bb},\bE)$ is the intersection cohomology with middle perversity on $\overline{Y}_K^{bb}$.
\end{theorem}

For our purpose, we shall also consider the Hodge, or mixed Hodge, structures on these cohomology groups. Here, as the local system $\bE$ underlies a natural variation of Hodge structure over $Y_K$ (cf.~\cite{Zu81}), $IH^k(\overline{Y}_K^{bb},\bE)$ carries a  mixed Hodge structure by Saito's \cite{Sa90} mixed Hodge module theory.  A priori the two Hodge structures need not correspond under the isomorphism but Harris and Zucker \cite[Theorem 5.4]{HZ01} nevertheless prove:
\begin{theorem} \label{TL2map}
	The composition
\begin{equation}\label{L2map}
\xi_k: H_{(2)}^k(Y_K,\bE)\cong IH^k(Y_K,\bE)\rightarrow H^k(Y_K,\bE)
\end{equation}
is a morphism of mixed Hodge structures. 
\end{theorem}
The image of $\xi_k$ is the lowest non-zero weight in the mixed Hodge structure of $H^k(Y_K,\bE)$ (see Remark 5.5 of \cite{HZ01}). In particular, the subspace $\mathrm{SC}^r_{\rm hom}(Y_K , \bE )$ lies in the image of $\xi_{2r}$. 

\begin{corollary} \label{HS}
We have
	\begin{equation}\label{L2-Zucker-Shimura}
	H^k_{(2)}(Y_K,\bE) \cong H^{k}(Y_K,\bE)
	\end{equation}
	as a morphism of Hodge structure for all $k<b-1$. 
\end{corollary}

\begin{proof}
The  isomorphism \eqref{L2-Zucker-Shimura} holds because the boundary of $\overline{Y}_K^{bb}$ has dimension at most one. See \cite[Remark 5.5]{HZ01} for the second statement.  
\end{proof}

\section{A surjectivity result for theta lifting}  
 In this section, we recap the work in \cite{BLMM16} on the surjectivity of global theta lifting for cohomological automorphic representations with trivial coefficients of orthogonal groups. We also explain how to extend to the case of nontrivial coefficients.  

\subsection{Global theta correspondence}
Here we briefly review Howe's global theta correspondence. Let us first fix some notations. Let $V$ be a non-degenerate quadratic space over $\QQ$ of signature $(2,b)$ and let $W$ be a symplectic space over $\QQ$ of dimension $2r$. We denote by $G'=\Mp_{2r}(W)$ the symplectic group $\Sp(W)$ if $b$ is even and the metaplectic double cover of $\Sp(W)$ if $b$ is odd. 

Fix a nontrivial additive character $\psi$ of $\AA/\QQ$. We denote by $\omega_\psi$ the (automorphic) Weil representation of $\mathrm{O}(V)(\AA) \times G'(\AA)$ realized in the space $\cS(V^{r}(\AA))$ of Schwartz-Bruhat  functions on $V^r(\AA)$ --- the so-called Schr$\ddot{\hbox{o}}$dinger model of the Weil representation. For each $\phi\in \cS(V^{r}(\AA))$ we form the theta function on $\mathrm{O}(V)(\AA) \times G'(\AA)$:
\begin{equation}\label{thetafunction}
\theta_{\psi,\phi}(g,g')=\sum\limits_{\xi \in V(\QQ)^r} \omega_\psi(g,g') (\phi)(\xi).
\end{equation}
Given an irreducible cuspidal automorphic representation $(\tau,H_\tau)$ of $G'(\AA)$ which occurs as an irreducible subspace of $L^2 (G' \QQ) \backslash G' \AA) )$ and given an element $f$ in that subspace, we can form the theta integral 
\begin{equation}\label{theta}
\theta_{\psi,\phi}^{f}(g)=\int\limits_{G'(\QQ)\backslash G'(\AA)}\theta_{\psi,\phi}(g,g') f(g') dg'.
\end{equation}
It is absolutely convergent and defines an automorphic function on $\mathrm{O}(V)(\AA)$, called the {\it global theta lift} of $f$.  
We shall denote by $\theta_{\psi,V} (\tau)$ the space of the automorphic representation generated by all the global theta lifts $\theta_{\psi,\phi}^f$ as $f$ and $\phi$ vary. We shall refer to the corresponding automorphic representation of $\mathrm{O}(V)(\AA)$ as {\it the global $\psi$-theta lift} of $\tau$ to $\mathrm{O}(V)$.  %This  yields the theta lifting for the special orthogonal group $G=\SO(V)$ as below. 

\begin{definition}
We say that a representation $\pi\in \cA(G)$ is in the image of the cuspidal $\psi$-theta correspondence from a smaller symplectic group if there exists a symplectic space $W$ with $\dim W \leq 2\left[ \frac{b+2}{2} \right]$ and an extension $\widetilde{\pi}$ of $\pi$ to $\mathrm{O}(V)$ such that $\widetilde{\pi}$ is   the global $\psi$-theta lift of an irreducible cuspidal automorphic representation of $G'=\Mp_{2r}(W)$,  i.e. there exists $\tau\in\cA_{\rm cusp}(G')$ such that  $\tilde{\pi}\hookrightarrow \theta_{\psi,V}(\tau)$. 
\end{definition}

\subsection{Cohomological representations} 
Throughout this section, we let $E$ be a finite dimensional irreducible representation of $G(\RR)$. Let $K_\RR \cong \SO (2) \times \SO (b)$ be as above. Denote by $\theta$ the corresponding Cartan involution and let $\mathfrak{g}_0 = \mathfrak{k}_0 \oplus \mathfrak{p}_0$ be the associated Cartan decomposition of the Lie algebra $\mathfrak{g}_0$ of $G(\RR )$. We fix a Cartan subalgebra $\mathfrak{t}_0 \subset \mathfrak{k}_0$. We shall denote by $\mathfrak{a}$ the complexification of a real Lie algebra $\mathfrak{a}_0$.

A unitary representation $\pi_\RR$ of $G(\RR)$ is {\it cohomological} (with respect to $E$) if its associated  $(\frg,K_\RR)$-module $\pi_\RR^{\infty}$ has nonzero relative Lie algebra cohomology, i.e. 
\begin{equation} \label{cohomrep}
H^\bullet (\frg,K_\RR ; \pi_\RR^\infty \otimes E) \neq 0.
\end{equation}

The unitary $(\frg,K_\RR)$-modules with nonzero cohomology have been classified by Vogan and Zuckerman \cite{VZ84}. They are determined by $\theta$-stable parabolic subalgebras $\mathfrak{q} \subset \mathfrak{g}$: $\mathfrak{q} = \mathfrak{l} \oplus \mathfrak{u}$, where $\mathfrak{l}$ is the 
centralizer of an element $X \in i \mathfrak{t_0}$ and $\mathfrak{u}$ is the span of the positive roots of $X$ in $\mathfrak{g}$. Then the Lie algebra $\mathfrak{l}$ is the complexification of $\mathfrak{l}_0 = \mathfrak{l} \cap \mathfrak{g}_0$ and we let $L$ be the connected subgroup of $G(\RR)$ with Lie algebra $\mathfrak{l}_0$. 
Associated to $\mathfrak{q}$, there is a well-defined, irreducible representation $V(\mathfrak{q})$ of $K_\RR$ that occurs with multiplicity one in $\wedge^R \mathfrak{p}$ where $R = \dim (\mathfrak{u} \cap \mathfrak{p} )$. Now if \eqref{cohomrep} holds there exists some $\mathfrak{q}$ such that the Cartan product of $V(\mathfrak{q})$ and $E^*$ occurs with multiplicity one in $\wedge^R \mathfrak{p} \otimes E^*$ and the $(\mathfrak{g} , K_\RR )$-module $\pi_\RR^\infty$ is isomorphic to some irreducible $(\mathfrak{g} , K_\RR )$-module $A_{\mathfrak{q}} (E)$ characterized by the following two properties:
\begin{enumerate}
\item $A_{\mathfrak{q}} (E)$ is unitary with the same infinitesimal character as $E$.
\item The Cartan product of $V(\mathfrak{q})$ and $E^*$ occurs (with multiplicity one) in $A_{\mathfrak{q}} (E)$. 
\end{enumerate}

In our case $K_\RR = \SO (2) \times \SO (b)$ acts on $\mathfrak{p} = (\C^2)^* \otimes \C^b$ through the standard representation of $\SO (2)$ on $\C^2$ and the standard representation of $\SO(b)$ on $\C^b$. We denote by $\C^+$ and $\C^-$ the $\C$-span of the vectors $e_1+ie_2$ and $e_1-ie_2$ in 
$\C^2$. The two lines $\C^+$ and $\C^-$ are left stable by $\SO(2)$. This yields a decomposition 
$\mathfrak{p} = \mathfrak{p}^+ \oplus \mathfrak{p}^-$ which corresponds to the decomposition given by the natural complex structure on $\mathfrak{p}_0$. For each non-negative integer $p$ the $K_\RR$-representation $\wedge^p \mathfrak{p} =
\wedge^p (\mathfrak{p}^+ \oplus \mathfrak{p}^-)$ decomposes as the sum:
$$\wedge^p \mathfrak{p} = \bigoplus_{r+s=p}  \wedge^r \mathfrak{p}^+ \otimes \wedge^r \mathfrak{p}^-.$$
The $K_\RR$-representations $\wedge^r \mathfrak{p}^+ \otimes \wedge^s \mathfrak{p}^-$ are not irreducible in general: there is 
at least a further splitting given by the Lefschetz decomposition:
$$\wedge^r \mathfrak{p}^+ \otimes \wedge^s \mathfrak{p}^- = \bigoplus_{k=0}^{\min (r,s)} \tau_{r-k , s-k}.$$
One can check that for $2(r+s) <b$ each $K_\RR$-representation $\tau_{r,s}$ is irreducible.  
Moreover in the range $2(r+s) <b$  only those with $r=s$ can occur as a $K_\RR$-type $V(\mathfrak{q})$ associated to a cohomological module. 
In the special case $r=r$ one can moreover check that each $\tau_{r,r}$ 
is irreducible as long as $r < b$; it is isomorphic to some $V(\mathfrak{q})$ where the Levi subgroup $L$ associated to $\mathfrak{q}$ is isomorphic to $C \times \SO_0 (2 , b-2r )$ with $C \subset K_\RR$. 

It follows in particular that if $\pi_\RR$ is an irreducible unitary representation of $G(\RR )$ that satisfies 
$$H^{r,r} (\frg,K_\RR ; \pi_\RR^\infty \otimes E) \neq 0$$
for some $r<b$ then $\pi_\RR^\infty$ is isomorphic to the unique unitary $(\mathfrak{g} , K_\RR )$-module that has the same infinitesimal character as $E$ and contains the Cartan product of $\tau_{r,r}$ and $E^*$. We shall denote by $A_{r,r} (E)$ this $(\mathfrak{g} , K_\RR )$-module. We have:
$$H^{i,j} (\frg , K ; A_{r,r} (E) \otimes E) = \left\{ 
\begin{array}{ll}
\C & \mbox{ if } r \leq i=j \leq  b -r, \ 2i \neq b \\
\C + \C & \mbox{ if } 2i=2j=b \\ 
0 & \mbox{ otherwise}.
\end{array} \right.
$$
We refer to \cite[\S 5.2 and 5.4]{BMM16} for more details. 

Note that $A_{r,r} (E)$ can only contribute to even degree cohomology. The following vanishing result therefore follows from the above classification (and Matsushima's formula \eqref{decomposition} below).

\begin{proposition} \label{oddvanish}
For any {\rm odd} degree $i< b/2$ and any local system $\bE$ we have 
$$H^{i} (Y_K,\bE)=0.$$
\end{proposition}
\begin{proof} 
Since $i < b/2$ implies $i<b-1$ we have $H^{i} (Y_K,\bE) \cong H_{(2)}^{i} (Y_K,\bE)$. The proposition then follows from the fact that in the range $r+s < b/2$ the only $K_\RR$-types $\tau_{r,s}$ than can occur in a cohomological representation are the ones with $r=s$. 
\end{proof}

\subsection{Surjectivity of the theta lift}
The following theorem is proved in \cite[Theorem 7.7]{BLMM16} when $E$ is the trivial representation. It was proved in \cite{BMM16} for general $E$ but under the hypothesis that $\pi$ is cuspidal. To be able to deal with residual representations as well was the main input of \cite{BLMM16}. Though the results of \cite{BLMM16} only addresses the case where $E$ is the trivial representation, the proofs (see especially the key Proposition 6.2) only make use of the fact that cohomological representations have integral regular infinitesimal character which is still true for representations that are cohomological with respect to $E$. We shall therefore not repeat the proof and simply refer to \cite{BLMM16} for the proof of:

\begin{theorem}\label{surjoftheta}
	Let $\pi=\otimes \pi_v\in \cA(G)$ be a square integrable automorphic representation of $G$. Suppose that the $(\mathfrak{g} , K_\RR )$-module $\pi_\RR^\infty$ of the local Archimedean component of $\pi$ is isomorphic to some cohomological module $A_{r,r} (E)$ with $3r<b+1$. Then there exists a cuspidal representation $\tau$ of $\Mp_{2r}(\AA)$ such that $\pi$ (up to a twist by a quadratic character) is in the image of the theta lift of $\tau$.
\end{theorem}

\subsection{Theta classes in cohomology groups} We keep all notations as in \S5.1. We  construct some special cohomology classes (w.r.t the local system $E$) on $Y_K$ from the relative Lie algebra cohomology of $(\frg,K_\RR )$-modules. We write
$$\pi=\pi_{\RR } \otimes\pi_f \in \cA_2(G),$$ 
and let $\pi^K_f$ be the finite dimensional subspaces of $K$-invariant vectors in $\pi_f$.  According to Matsushima's formula and Langlands spectral decomposition (see~\cite{BC83}, \cite[\S7.8]{BLMM16}), we have   
\begin{equation}\label{decomposition}
H^i_{(2)}(X_K, \bE)\cong \bigoplus\limits_{\pi\in \cA(G)} m(\pi) H^i(\frg,K_\RR ; \pi_\RR^\infty \otimes E)\otimes \pi_f^K.
\end{equation}
where $\pi$ occurs discretely in  $L^2(G(\QQ)\backslash G( \AA ))$  with multiplicity $m(\pi)$.  Following \cite{BLMM16}, we define the  space of theta classes $$H^i_\theta(X_K,\bE)\subseteq H^{i}(X_K,\bE)$$ as the subspace generated by the image of $H^i(\frg, K_\RR ; \pi_\RR^\infty\otimes E)$ via \eqref{decomposition} and \eqref{L2map},  where $\pi$ varies among the irreducible representations in $\cA(G)$ which are in the image of the $\psi$-cuspidal theta correspondence from a smaller symplectic group.  The space $H^i_\theta(Y_K,\bE)$ of theta classes on $Y_K$ are naturally defined by restriction. The main result in \cite{BLMM16} can be reformulated as:

\begin{theorem} \label{T:531}
	There is an inclusion 
	\begin{equation}\label{surjoftheta}
	H^{2r}_\theta(Y_K,\bE)\subseteq \mathrm{SC}^{r}_{\rm hom}(Y_K,\bE).
	\end{equation}
Moreover,
	$H^{2r}_\theta(Y_K,\bE)=H^{r,r}(Y_K,\bE)$ when $r<\frac{b+1}{3}$ or $r>\frac{2b-1}{3}$.
\end{theorem}

\begin{proof}
The inclusion \eqref{surjoftheta} follows from (the proof of) \cite[Proposition 11.3]{BMM16}.\footnote{Beware that our space $\mathrm{SC}^\ast$ is a priori larger than the one defined in \cite{BMM16}.} The last assertion follows from Theorem \ref{surjoftheta}, the decomposition \eqref{decomposition} and the fact that if $\pi_\RR$ is an irreducible unitary representation of $G(\RR )$ that satisfies 
$$H^{r,r} (\frg,K_\RR ; \pi_\RR^\infty \otimes E) \neq 0$$
for some $r<b$ then $\pi_\RR^\infty$ is isomorphic to $A_{r,r} (E)$. Note that the fact that $H^{2r}_\theta(Y_K,\bE)$ is contained in $H^{r,r}(Y_K,\bE)$ obviously follows from the inclusion \eqref{surjoftheta}. It also follows from the fact that the only cohomological representations that are in the image of the local Archimedean theta correspondence are the $A_{r,r} (E)$, see~\cite[\S7.2]{BLMM16} and \cite{Li}. 
\end{proof}

Let $H^{2k}_{\rm alg}(Y_K,\CC)$ be the image of $\CH^k(Y_K)\otimes \CC$ in $H^{2k}(Y_K,\CC)$ via the cycle class map.  Following \cite{Dan16}, one can use Zucker's conjecture and the hard Lefschetz theorem on intersection cohomology to get the following result

\begin{corollary} \label{Cpet}
When $b>3$, we have
\begin{equation}\label{1cycle}
H^{2b-2}_{alg}(Y_K,\CC)=0.
\end{equation}
\end{corollary}

\begin{proof}Let us quickly sketch the proof. Let $\lambda$ be the Hodge line bundle on $Y_K$.  
By \cite{BLMM16} and Zucker's conjecture, we have that
\begin{equation}
IH^2(\overline{Y}^{bb}_K,\CC)\cong H^2(Y_K,\CC)
\end{equation}
 is generated by connected cycles of codimension one. Next, as $\lambda$ is ample, the Hard Lefschetz theorem for intersection cohomology implies that the image of the map 
 \begin{equation}
\xi_{2b-2} :IH^{2b-2}(\overline{Y}^{bb}_K,\CC)\rightarrow H^{2b-2}(Y_K,\CC)
 \end{equation}
is spanned by the class of $\lambda^{b-2}\cdot c(U,g,K)$.
By \cite{GT05}, the class $\lambda^{b-2}\cdot c(U,g,K)$ is zero in $H^{2b-2}(Y_K,\CC)$ when $\dim c(U,g,K)\geq 3$. Since  $H^{2b-2}_{alg}(Y_K,\CC)$ lies in the image of $\xi_{2b-2}$, this proves the assertion. 
\end{proof}

Let us finally define theta classes in the cohomology groups of the moduli spaces discussed in Section 3. 
Assume that $E$ is a finite dimensional representation of $\SO(\Lambda_\RR )$.  Recall that there is a period map $\cP_\Sigma : \cF_{\Sigma,h}\rightarrow \Gamma_{\Sigma}\backslash D_\Sigma$.  Let $\EE=\cP_\Sigma^\ast (\bE)$ be the pullback of $\bE$ to $\cF_{\Sigma,h}$. Then we define
 \begin{equation}
H^\bullet_\theta(\cF_{\Sigma,h},\EE):=\cP_\Sigma^\ast (H_\theta^\bullet(\Gamma_{\Sigma }\backslash D_\Sigma ,\bE)).
 \end{equation}
as the subspace of theta classes on $\cF_{\Sigma,h}$.
It follows from Theorem \ref{T:531} that:
\begin{corollary}\label{surjmoduli} There is an inclusion 
	$$H^{2r}_\theta(\cF_{\Sigma,h}, \EE)\subseteq \cP_\Sigma^\ast (\mathrm{SC}^r_{\rm hom} (\Gamma_{\Sigma}\backslash D_\Sigma ,\bE ))$$ for all $r$. 
\end{corollary}

\section{The Funke-Kudla-Millson ring} 

In this section, we introduce special Schwartz forms at the Archimedean place and Kudla-Millson's special theta lift.  Following \cite{BMM16} and \cite{BLMM16}, this establishes a connection between special cycle classes and theta classes as defined in \S5.3. More importantly for us, we show that this yields a so called {\it Funke-Kudla-Millson ring}, or just FKM ring, in the space of differential forms with coefficients, which plays the key role in our study of the cohomological tautological conjecture.

\subsection{Special Schwartz forms}
In this subsection we let $V$ be a real quadratic space of signature $(b,2)$ of dimension $m=b+2$. We use signature $(b,2)$ rather than signature $(2,b)$ in order to follow the notations of the works of Kudla-Millson and Funke-Millson; when applied to our geometric situation we will reverse the variables.  

Pick an oriented orthogonal basis $\{ v_i \}$ of $V$ such that $(v_\alpha , v_\alpha ) = 1$ for $\alpha = 1 , \ldots, b$ and $(v_\mu , v_\mu ) = -1$ for $\mu=b+1 , b+2$. We shall use the notations of \cite{BMM16}.

Kudla and Millson \cite{KM90} and then Funke and Millson \cite{FM06} have constructed special Schwartz forms $\varphi$ in 
$$\left[ \mathcal{S} (V^r) \otimes A^{2r} (D) \otimes T^\ell (V) \right]^{\mathrm{O} (V)} \cong \mathrm{Hom}_{K_\RR } (\wedge^{2r} \mathfrak{p} , \mathcal{S} (V^r) \otimes T^\ell (V) ).$$
By abuse of notations we also denote by $T^\ell (V)$ the local system on $D$ associated to $T^{\ell} (V)$. If $\mathbf{x}  \in V^r$ we then have $\varphi (\mathbf{x}) \in A^{2r} (D, T^\ell (V))$. 

We shall rather work with the polynomial Fock model for the dual pair $\mathrm{O} (b,2) \times \mathrm{Sp}_{2r} (\mathbb{R})$, see \cite[Section 7]{BMM16}. We denote by $\iota$ the intertwining operator from the Sch\"odinger model to the Fock model, see \cite[Section 6]{KM90}. The map $\iota$ maps the vectors of $\mathcal{S} (V^r )$ that are finite under the action of a maximal compact subgroup $\mathrm{U} (2rm)$ of the symplectic group $\mathrm{Sp}_{2rm}$ containing the dual pair onto the space of polynomials 
$$\mathcal{P} (\mathbb{C}^{mr}) \cong \mathbb{C}[z_{1 , j} , \ldots , z_{m, j} \; : \; j=1, \ldots r ].$$
Here if $\mathbf{x} = (x_1 , \ldots , x_r ) \in V^r$ we have: 
$$x_j = \sum_{\alpha  =1}^b z_{\alpha , j}  v_\alpha  + \sum_{\mu=b+1}^{b+2} z_{\mu , j} v_\mu . $$
We will use the notation $\mathcal{P} (\mathbb{C}^{mr})_+$ to denote the polynomial in the ``positive'' variables $z_{\alpha , j}$, $1 \leq \alpha \leq b$, $1 \leq j \leq r$. We have an isomorphism of $K_\infty \times K_\infty '$-modules 
$$\mathcal{P} (\mathbb{C}^{mr})_+ \cong \mathrm{Pol} (M_{b,r} (\mathbb{C} )).$$
Here $K_\infty \cong \mathrm{O} (b) \times \mathrm{O} (2)$ is a maximal compact subgroup of $\mathrm{O} (b,2)$, the groups $K_\infty '$ is a maximal compact subgroup of $\mathrm{Sp}_{2r} (\mathbb{R})$ and the action of the group $\mathrm{O} (b)$ on $\mathrm{Pol} (M_{b,r} (\mathbb{C} ))$ is induced by the natural action on the columns (i.e., from the left) of the matrices. 

First consider the case $r=1$. Then we simply set $z_\alpha = z_{\alpha , 1}$. Note that $\mathfrak{p} \cong \mathbb{C}^b \otimes \mathbb{C}^2$; we let $\omega_{\alpha , \mu}$ be the linear form which maps an element of $\mathfrak{p}$ to its $(\alpha, \mu)$-coordinate. For any multi-indices $\underline{\alpha } = (\alpha_1 , \alpha_2 )$ we write $\omega_{\underline{\alpha}} = \omega_{\alpha_1 , b+1} \wedge \omega_{\alpha_2 , b+2}$ and $z_{\underline{\alpha}} = z_{\alpha_1 } z_{\alpha_2}$. The form 
$$\sum_{\underline{\alpha}} z_{\underline{\alpha}} \otimes \omega_{\underline{\alpha}} \in \mathrm{Hom}_{K_\infty } (\wedge^2 \mathfrak{p} , \mathcal{P} (\mathbb{C}^{m})_+ )$$ 
is precisely the image $\iota (\varphi_{1,0})$ of the Schwartz form $\varphi_{1,0} \in \mathrm{Hom}_{K_\infty } (\wedge^{2} \mathfrak{p} , \mathcal{S} (V))$, constructed by Kudla and Millson, under the intertwining operator $\iota$. The natural product
$$\bigotimes_{j=1}^r \mathbb{C}[z_{1 , j} , \ldots , z_{b, j} ] \rightarrow \mathbb{C}[z_{1 , j} , \ldots , z_{b, j} \; : \; j=1, \ldots r ]$$
induces a natural map 
$$\mathrm{Hom}_{K_\infty } (\wedge^2 \mathfrak{p} , \mathcal{P} (\mathbb{C}^{m})_+ ) \otimes \cdots \otimes \mathrm{Hom}_{K_\infty } (\wedge^2 \mathfrak{p} , \mathcal{P} (\mathbb{C}^{m})_+ ) \to \mathrm{Hom}_{K_\infty } (\wedge^{2r} \mathfrak{p} , \mathcal{P} (\mathbb{C}^{mr})_+ )$$
which maps 
$$\otimes_{j=1}^r (\sum_{\underline{\alpha}} z_{\underline{\alpha} , j} \otimes \omega_{\underline{\alpha}} )$$
onto $\iota (\varphi_{r,0} )$ the image of the Kudla-Millson form $\varphi_{r,0} \in \mathrm{Hom}_{K_\infty } (\wedge^{2r} \mathfrak{p} , \mathcal{S} (V^r))$. We shall abusively denote it $\varphi_{r,0}$ as well. 

We now describe Funke-Millson forms. For any multi-indices $\underline{\alpha } = (\alpha_1 , \ldots , \alpha_\ell )$ we let 
$$v_{\underline{\alpha}} = v_{\alpha_1} \otimes \cdots \otimes v_{\alpha_\ell} \in T^\ell (V ).$$
Given any $\ell$-tuple $I= (i_1 , \ldots , i_{\ell})$ of integers in $[1,b]$ we define 
$$\varphi_{0,\ell}^{I} \in \mathrm{Hom} (\mathbb{C} , [\mathcal{P} (\mathbb{C}^{mr})_+ \otimes T^{\ell} (V) ]^{K_\infty } )$$
by 
$$\varphi_{0,\ell}^I = \sum_{\underline{\alpha}} (z_{\alpha_1 , i_1} \cdots z_{\alpha_\ell , i_\ell} ) \otimes v_{\underline{\alpha}} .$$
Using the natural product on $\mathcal{P} (\mathbb{C}^{mr})_+$,  we have 
$$\varphi_{r,\ell}^I = \varphi_{r,0} \cdot \varphi_{0,\ell}^I \in \mathrm{Hom}_{K_\infty } (\wedge^{2r} \mathfrak{p} , \mathcal{P} (\mathbb{C}^{mr})_+ \otimes T^{\ell} (V) ).$$
We shall denote $\Phi_{T^\ell (V)}$ the subspace of $\mathrm{Hom}_{K_\infty } (\wedge^{2r} \mathfrak{p} , \mathcal{P} (\mathbb{C}^{mr}) \otimes T^{\ell} (V) )$ spanned by the Schwarz forms $\varphi_{r,\ell}^I$.

Now fix a partition $\lambda = (\lambda_1 , \ldots , \lambda_k)$, with $k = \left[ \frac{m}{2} \right]$, and let $\ell = \lambda_1 + \ldots + \lambda_k$. The finite dimensional $\mathrm{O} (V)$-module $\mathbb{S}_{[\lambda]} (V)$ can be obtained as the image of the classical Schur functor $\mathbb{S}_{\lambda} (V) \subset T^{\ell} (V)$ under the $\mathrm{O}(V)$-equivariant projection of $T^\ell (V)$ onto the harmonic tensors. We denote by $\pi_{[\lambda ]}$ the 
corresponding $\mathrm{O}(V)$-equivariant projection of $T^\ell (V)$ onto $\mathbb{S}_{[\lambda]} (V)$. We set 
$$\varphi_{r , [\lambda]}^I = (1\otimes \pi_{[\lambda ]}) \circ \varphi_{r ,\ell}^I.$$
We shall denote $\Phi_{[\lambda]}^{2r}$ the subspace of $\mathrm{Hom}_{K_\infty } (\wedge^{2r} \mathfrak{p} , \mathcal{P} (\mathbb{C}^{mr}) \otimes T^{\ell} (V) )$ spanned by the Schwarz forms $\varphi_{r,[\lambda]}^I$.

More generally if $E$ is {\it any} finite dimensional representation, it decomposes into a finite direct sum (with multiplicities) of $\mathrm{O} (V)$-modules 
$\mathbb{S}_{[\lambda ]} (V)$ and we shall denote $\Phi_E^{2r}$ the finite direct sum of the corresponding $\Phi_{[\lambda]}^{2r}$'s identified as a subspace of $\mathrm{Hom}_{K_\infty } (\wedge^{2r} \mathfrak{p} , \mathcal{P} (\mathbb{C}^{mr}) \otimes E )$. We will in fact abusively identify $\Phi_E^{2r}$ with a subspace 
$$\Phi_E^{2r} \subset \left[ \mathcal{S} (V^r ) \otimes A^{2r} (D) \otimes E \right]^{\mathrm{O}(V)}$$
using the intertwining operator $\iota$.

\subsection{The Funke-Kudla-Millson ring}
Consider a finite graded ring of finite dimensional representations $\mathbf{H}^\ast$ of $\mathrm{O} (V)$. We use the above special Schwartz forms to construct a subring of the bi-graded complex 
\begin{equation} \label{ring}
\bigoplus_{i,j} \left[ \mathcal{S} (V^i ) \otimes A^{2i} (D) \otimes \mathbf{H}^j \right]^{\mathrm{O}(V)}. 
\end{equation}
Indeed: the wedge product
$$\wedge : A^{2i} (D,\mathbf{H}^j) \times A^{2i'} (D,\mathbf{H}^{j'}) \to A^{2(i+i')} (D, \mathbf{H}^{j} \otimes \mathbf{H}^{j'})$$
composed with the product $\mathbf{H}^{j} \otimes \mathbf{H}^{j'} \to \mathbf{H}^{j+j'}$ yields a ring structure on \eqref{ring} and since the latter obviously maps
$\Phi_{\mathbf{H}^j}^{2i} \times \Phi_{\mathbf{H}^{j'}}^{2i'}$ into $\Phi_{\mathbf{H}^{j+j'}}^{2(i+i')}$ we conclude:

\begin{proposition} \label{P:ring}
The subspace
$$\bigoplus_{i,j} \Phi_{\mathbf{H}^j}^{2i} \subset \bigoplus_{i,j} \left[ \mathcal{S} (V^i ) \otimes A^{2i} (D) \otimes \mathbf{H}^j \right]^{\mathrm{O}(V)}$$
is a subring.
\end{proposition}

Consider now a (global) arithmetic quotient $Y_K$ with $K \subset \mathrm{O} (V) (\mathbb{A}_f )$ an open compact subgroup. Let $\varphi \in \mathcal{S} (V (\mathbb{A}_f )^i )$ be a $K$-invariant Schwartz function for some $i$. Then for any $\varphi_\RR \in \Phi_{\mathbf{H}^j}^{2i}$, we define a global Schwartz form
$$\phi = \varphi_\RR \otimes \varphi \in \left[ \mathcal{S} (V (\mathbb{A} )^i ) \otimes A^{2i} (D) \otimes \mathbf{H}^j \right]^{\mathrm{O}(V)}.$$
Applying the theta distribution \eqref{thetafunction} to $\phi$ yields a theta function $\theta_{\psi , \phi} (g, g')$ which, as a function of $g \in \mathrm{O} (V)$, defines a differential form in $A^{2i} (X_K , \mathbf{H}^j )$. Varying $\varphi$ and $g'$ we obtain a subspace 
$$A_{\rm special}^{2i} (X_K , \mathbf{H}^j ) \subset A^{2i} (X_K , \mathbf{H}^j ).$$
It follows from Proposition \ref{P:ring} that
\begin{equation}
\Phi_{\mathbf{H}^\ast} := \bigoplus_{i,j} A_{\rm special}^{2i} (X_K , \mathbf{H}^j ) 
\end{equation}
is a subring of the graded ring $\bigoplus_{i,j} A^{2i} (X_K , \bH^j )$; we shall refer to it as the {\it Funke-Kudla-Millson ring} of $X_K$. Restricting to the connected component $Y_K \subset X_K$ yields the Funke-Kudla-Millson ring of $Y_K$.

\subsection{Relation with $\theta$-classes}	
By construction, the cohomology class of a differential form in $A_{\rm special}^{2i} (Y_K , \mathbf{H}^j )$ is a $\theta$-class in 
$$H_{\theta}^{i,i} (Y_K , \mathbf{H}^j ) := H_{\theta}^{2i} (Y_K , \mathbf{H}^j ) \cap H^{i,i} (Y_K , \mathbf{H}^j ).$$
It turns out that the converse is also true:

\begin{proposition}[see Theorem 11.2 of \cite{BMM16}] \label{FKMtheta}
Any class in $H_{\theta}^{i,i} (Y_K , \mathbf{H}^j )$ can be represented by a differential form in $A_{\rm special}^{2i} (Y_K , \mathbf{H}^j )$.
\end{proposition}

\section{Universal family of polarized hyperk\"ahler manifolds }

In this section, we study the cohomology groups of universal family of polarized hyperk\"ahler manifolds via Deligne's decomposition theorem and prove the cohomological Franchetta conjecture. Furthermore,  we prove the cohomological tautological conjecture using the FKM ring. Throughout this section, we concern with hyperk\"ahler manifolds with type $\Lambda$ and $\rank ~\Lambda=3+b$. 

\subsection{Cohomology of universal families of hyperk\"ahler manifolds}\label{cohfran}
Let $\pi:\cU\rightarrow \cF$ be a smooth family of  polarized hyperk\"ahler manifolds over a variety $\cF$.  Let $\HH_\pi^i $ denote the local system on $\cF$ whose fiber at a point $p\in \cF$ is $H^i(\cU_p,\QQ)$, with $\cU_p=\pi^{-1}(p)$. 
According to the decomposition theorem of Deligne \cite{De68}, we have
\begin{equation}\label{Dec}
H^k(\cU,\QQ)\cong \bigoplus\limits_{i+j=k} H^i(\cF; \HH^j_\pi)
\end{equation}
compatible with the mixed Hodge structure on both sides (cf.~\cite[$\S$4.3]{deM12}). 

For our purpose, we shall consider the universal family (as stacks) $$\pi_\ell:\cU^\ell_h\rightarrow\cF^\ell_h$$ of polarized hyperk\"ahler manifolds of type $\Lambda$ and  dimension $2n$ with a full $\ell$-level structure defined in \S\ref{Slevel}.  The decomposition theorem still applies in this case.  Our first result is a cohomological version of O'Grady's generalized Franchetta conjecture for hyperk\"ahler manifolds.  Here we work with stack cohomology. 

\begin{theorem}\label{mainthm2} 
Let $m_\Lambda$ be the second Betti number of hyperk\"ahler manifolds in $\cF_h^\ell$.
Assume that $n < \frac{m_\Lambda-3}{8} $.  Then, for any class $\alpha\in \CH^{2n}(\cU_{h}^\ell ,\QQ)$, there exists $m\in \QQ$ such that $[\alpha-mc_{2n}(\cT_{\pi_\ell})] \in H^{4n} (\cU^\ell_{h},\QQ)$ supports on the Noether-Lefschetz locus of $\cF^\ell_{h}$. In particular, Theorem \ref{main3} holds.
\end{theorem}
\begin{proof} Obvisouly, we can assume  $\ell$ is sufficiently large.  The local system $\HH^i_{\pi_\ell}=R^i({\pi_\ell})_\ast \QQ$ on $\cF_h^\ell$ descends to its coarse moduli space $F_h^\ell$ and we write it  by $\HH^i_{\pi_\ell}$ by abuse  of notations.   There is a natural isomorphism 
\begin{equation}\label{stackcoh}
H^k(\cF^\ell_h,  \HH^i_{\pi_\ell})\cong H^k(F^\ell_h, \HH^i_{\pi_\ell})
\end{equation}	
and we may regard $F_h^\ell$ as an open subset of $\Gamma_h^\ell\backslash D$.

Next, under our hypothesis, Proposition \ref{oddvanish} implies that, in the degrees $i \leq 4n$ we are concerned with, only {\it even} degree contribute non-trivially to the algebraic part of the decomposition \eqref{Dec}. Since  moreover $H^0(\cF_{h}^\ell,\HH^{4n}_{\pi})\cong \QQ$, it follows that there exists $m\in\QQ$ such that 
\begin{equation}\label{fran}
[\alpha-mc_{2n} (\cT_{\pi_\ell})]\in \bigoplus_{i=1}^{2n} H^{2i}(\cF_{h}^\ell, \HH^{2(2n-i)}_{\pi_\ell}).
\end{equation}
Denote by $\gamma_i\in H^{2i}(\cF_{h}^\ell, \HH^{2(2n-i)}_{\pi_\ell} )$ the $i$-th component of \eqref{fran}. Since  $[\alpha-mc_{2n}(\cT_\pi)]\in H^{4n}(\cU_{h}^\ell , \QQ)$ is algebraic,  it lies in the lowest  weight subspace of $H^{4n}(\cU_{h}^\ell,\QQ )$.  As \eqref{Dec} is compatible with the mixed Hodge structure on both sides, each $\gamma_i$ lies in the lowest weight subspace of   $H^{2i}(\cF_{h}^\ell, \HH^{2(2n-i)}_{\pi_\ell})$. 

Recall that Lemma \ref{Descend}  implies that there exists an automorphic local system $$\bH^\bullet=\bigoplus \bH^j$$ on $\Gamma_{h}^\ell \backslash D$ such that  $\HH^j_{\pi_\ell}=R^j ({\pi_\ell})_\ast \QQ$ on $\cF_{h}^\ell$ is the pullback $(\mathcal{P}^\ell_h)^\ast \bH^j$.   So we have a natural map  
$$(\mathcal{P}^\ell_h )^* : H^{2i}(\Gamma_{h}^\ell \backslash D , \bH^{4n-2i}) \to H^{2i}(F_{h}^\ell, \HH^{4n-2i}_{\pi_\ell }) \cong H^{2i}(\cF^\ell_h,\HH^{4n-2i}_{\pi_\ell })$$ 
induced by $\cP^\ell_h$, which is a mixed Hodge structure morphism,  and it is surjective onto the non-zero lowest weight part. We can therefore lift $\gamma_i$ to an element 
$$\tilde{\gamma}_i\in  H^{2i}(\Gamma_h^\ell \backslash D , \bH^{4n-2i}), $$ 
with lowest weight. Now it suffices to show 
\begin{equation}\label{dcomponent}
\tilde{\gamma}_i\in \mathrm{SC}^i_{\rm hom} (\Gamma_h^\ell \backslash D, \bH^{2(2n-i)})
\end{equation}
since elements in $\mathrm{SC}^i_{\rm hom} (\Gamma_h^\ell \backslash D, \bH^{2(2n-i)})$ obviously support on proper Shimura subvarieties.

However, as the local system $\bH^\ast $ arises from a finite dimensional representation of $\mathrm{O}(V)$,  Theorem \ref{T:531} can be applied and this directly implies \eqref{dcomponent}.  
In particular, the conditions in Theorem \ref{mainthm2} hold for K3$^{[n]}$-type hyperk\"ahler manifolds with  $n\leq 2$ and hence Theorem \ref{main3} holds.
\end{proof}

\subsection{Leray spectral sequence and cup product}\label{S8.2}
There are natural cup products on both sides of \eqref{Dec}. In general these coincide only on the associated graded rings, see e.g. \cite{BottTu}. However Voisin \cite{Vo12} has proved that these two cup products are 
compatible for K3 surfaces after shrinking to a smaller open subset. In what follows we partially extend her results to families of hyperk\"ahler manifolds. 

Let us first make some conventions. In this subsection, we will deal with the universal family of lattice polarized hyperk\"ahler manifolds with level structures (See $\S$3.3 for all notions). For simplicity of the notations,  we will use $\pi:\cU \rightarrow \cF$ to denote  the universal family $\cU_{\Sigma,h}^\ell\rightarrow\cF_{\Sigma,h}^\ell$ of $h$-ample $\Sigma$-polarized hyperk\"ahler manifolds with a full $\ell$-level structure and let $Y= \Gamma_{\Sigma}^\ell\backslash D_\Sigma$. We denote by 
$$\cP:\cF\rightarrow Y$$ 
the period map.  Moreover, we choose $\ell$  sufficiently large so that $Y$ is smooth and  Lemma \ref{Descend} applies.

Consider on $\cU$ the short exact sequence of vector bundles 
$$0 \to \pi^* \Omega_{\cF} \to \Omega_{\cU} \to \Omega_{\cU / \cF} \to 0$$
that defines the fiber bundle of relative differential forms $\Omega_{\cU / \cF}$. This gives a decreasing filtration $L^* \Omega^q_{\cU }$ of the fiber bundle $\Omega^q_{\cU}$ defined by 
$$L^p \Omega^q_{\cU} = \pi^* \Omega^p_{\cF } \wedge \Omega^{q-p}_{\cU }.$$
The associated graded vector bundle is $\mathrm{Gr}^p_L \Omega^q_{\cU} =  \pi^* \Omega^p_{\cF } \otimes \Omega^{q-p}_{\cU}$. This yields a filtration on the complex $A^q (\cU )$ which is the space of smooth sections of $\Omega^q_{\cU}$. 
We therefore have a corresponding (Leray) spectral sequence that computes the cohomology of $\cU$ from that of the base $\cF$ and the fiber. The $E_1$ term is 
$$E_1^{p,q} = A^p (\cF, R^q\pi_\ast \RR ).$$

Note that the filtration is compatible with the cup-product meaning that if $\alpha \in L^{p_1} A^{k_1} (\cU)$ and $\beta \in L^{p_2} A^{k_2} (\cU)$ then $\alpha \wedge \beta \in L^{p_1 + p_2} A^{k_1 + k_2} (\cU)$. The cup-pruduct on $ A^\ast (\cU)$ induces the natural structure on $E_1$. 
Recall  that $R\pi_\ast \RR$ is the pullback  of a local system $\bH^\ast$ on $Y$ that comes from a finite dimensional representation of $\mathrm{O} (\Sigma^\perp(\RR))$. The 
pullback $\cP^* \Phi_{\bH^\ast }$ of the FKM ring on $Y$ associated to $\Phi_{\bH^\ast}$ defines a subring of $E_1$; we will abusively refer to it as the FKM subring of the $E_1$ term of the Leray spectral sequence that computes the cohomology of $\cU$ from that of the base $\cF$ and the fiber.  

\begin{theorem} \label{Tcup}
With the notations as above, for any two cycles $\alpha_1,\alpha_2\in \CH^\ast(\cU)$ of codimension $< \frac{m_\Lambda-3}{4} $, the two cup-products of $[\alpha_1]$ and $[\alpha_2]$ in $H^\ast(\cU,\CC)$ associated to the two sides of \eqref{Dec} differ by a class supported on the Noether-Lefschetz locus of $\cF$.
\end{theorem}
\begin{proof} Let $2k_\ell$ be the degree of $\alpha_\ell$ ($\ell=1,2$). Denote by $\alpha_\ell^{(i)}$ the $2i$-th component of $\alpha_\ell$ in the decomposition \eqref{Dec}. Since $k_\ell< \frac{m_\Lambda-3}{4} $, the argument in the proof of Theorem \ref{mainthm2} implies that each $\alpha_\ell^{(i)}$ can be lifted as a $\theta$-class in $H_\theta^{2i} (Y, \bH^{2(k_\ell -i)})$. Now Proposition \ref{FKMtheta} implies that such a $\theta$-class can be represented by a differential form in the FKM ring. It follows that both $\alpha_1$ and $\alpha_2$ can be represented by elements of the FKM subring of the $E_1$ term of the Leray spectral sequence that computes the cohomology of $\cU$ from that of the base $\cF$ and the fiber. Now being a subring FKM is stable by cup-products. It follows that $\alpha_1 \wedge \alpha_2$ belongs to the FKM ring and that
$$\alpha_1 \wedge \alpha_2 - \alpha_1^{(0)} \wedge \alpha_2^{(0)} \in \bigoplus_{i=1}^{k_1+k_2} H^{i,i}_\theta (Y, \bH^{2(k_1+k_2-i)})$$
is supported on the Noether-Lefschetz locus of $\cF$ by Corollary \ref{surjmoduli}. 
\end{proof}

\subsection{Towards the Cohomological Tautological Conjecture} 
Let $\cU\rightarrow \cF$ be a smooth connected family of projective hyperk\"ahler manifolds. We denote by $R^\ast_\pi(\cF)\subseteq \CH^\ast(\cF)$ the subring generated by all $\kappa$-classes.   An easy fact is that inclusions
\begin{equation}
\NL^\ast_\pi(\cF) \subseteq R^\ast_\pi (\cF) 
\end{equation}
are preserved under the pullback by morphisms preversing the relative Picard group (modulo $\rm{DCH}^\ast(\cF)$).  Here,  a morphism between two families $\pi: \cU\rightarrow \cF$ and $\pi':\cU'\rightarrow \cF'$ preserving the relative Picard group is a commutative diagram
\begin{equation}
\xymatrix{ \cU' \ar[r]^{\bar{f}}\ar[d]^{\pi'} &\ \cU\ar[d]^{\pi}\\ \cF' \ar[r]^f & \cF  } 
\end{equation}
 with $\bar{f}^\ast \Pic(\cU/\cF)\cong \Pic(\cU'/\cF')$. This is because $f^\ast(\NL_\pi^\ast( \cF))\subseteq \NL^\ast_{\pi'}(\cF')$ and $f^\ast R_\pi^\ast(\cF)=R_{\pi'}^\ast( \cF')$ modulo $\rm{DCH}^\ast(\cF')$.  

Using this fact, we can work on $\kappa$-classes on a general family of hyperk\"ahler manifolds and Theorem \ref{main2} will be deduced from the following result as a direct consequence.

\begin{theorem}\label{tauthm}
Let $\pi:\cU\rightarrow \cF$ be a smooth family of  $h$-polarized hyperk\"ahler manifolds of type $\Lambda$ over an irreducible quasi-projective variety  $\cF$ of dimension $b \geq 3$.  Let $r+1$ be Picard number of the generic fiber of $\pi$ and  let 
$$\cB = \{ \cL_0 , \ldots , \cL_r \} \subset \mathrm{Pic}_\QQ (\cU)$$ 
be a collection of line bundles whose images in $\mathrm{Pic}_\QQ  (\cU / \cF)$ form a basis. If $b_j=0$ for $j\geq \frac14 (m_\Lambda-3)$, the $\kappa$-class $[\kappa_{a_0,\ldots, a_r,b_1,\ldots b_{2n} }^\cB]\in H^{\ast}(\cF,\QQ)$ is lying in $\NL^\ast_{\rm hom}(\cF)$. 
\end{theorem}

\begin{proof}
According to the discussion above, it suffices to show that the assertion holds when $\pi:\cU\rightarrow \cF$ is a connected component of  the universal family (as stacks)  of lattice-polarized hyperk\"ahler manifolds. Because of \eqref{E45}, we can add level structures and further be reduced to the case where $\pi:\cU\rightarrow \cF$ is the universal family $\cU_{\Sigma,h}^\ell\rightarrow\cF_{\Sigma,h}^\ell$ in \S\ref{S8.2}. 

 To prove the assertion, we only need to show that the cup product  
\begin{equation}\label{cups}
c_1^{a_i}(\cL_i)\ldots c_1(\cL_r)^{a_r}\ldots c_1(\cT_\pi)^{b_1}\ldots c_{2n}(\cT_\pi)^{b_{2n}}
\end{equation}
lies in the FKM-ring because the pushward of classes in FKM-ring are lying in $H^\ast_\theta(\cF,\CC)$. 
Then the argument in the proof of Theorem \ref{Tcup} implies that as long as $b>2$, it suffices to show $c_1(\cL_i)$ and $c_{j}(\cT_\pi)$ are lying in the FKM-ring on $\cU$. 

Using the notations in $\S$8.2,   the cycle classes of $c_1(\cL_i)$ and $c_{j}(\cT_\pi)$ can be lifted to classes in $\bigoplus\limits_{p+q=2} H^{p}(Y, \bH^q)$ and  $\bigoplus\limits_{p+q=2j} H^{p}(Y, \bH^q)$ respectively  via the period map $\cP: \cF\rightarrow Y$  (See the proof of Theorem \ref{mainthm2}). 

Next, by our assumption $b\geq 3$, Theorem \ref{T:531}  implies that the  FKM-ring on $Y$ contains $\bigoplus\limits_{p+q=2} H^{p}(Y, \bH^q)$ and $\bigoplus\limits_{p+q=2j} H^{p}(Y, \bH^q)$ when  $j<\frac{m_\Lambda-3}{4}$.   This proves our assertion.
 
Last, let us  explain how we get the bound of $j$. Note that Theorem \ref{T:531} actually only applies when  $j<\frac{\dim Y}{4}$. However, recall that the moduli space $\cF_{h}^\ell$ has dimension $m_\Lambda-3$, and $c_j(\cT_{\pi})$ can be lifted  to the universal family $\cU_{h}^\ell\rightarrow\cF_{h}^\ell$ as $\cT_\pi$ descends. The lifted class lies in the  FKM-ring of $\cU_{h}^\ell$ by Theorem \ref{T:531} and the same argument as above. Then our claim follows from the fact the pullback of FKM-ring on $\cU_{h}^\ell$ is contained  in the FKM-ring of $\cU$. 
\end{proof}

\begin{remark}\label{fail}
Theorem \ref{tauthm} is no longer true  when $ \dim \cF\leq 2$. For instance,  there exist divisors on Hilbert modular surfaces  which are not in the span of modular curves. (cf.~\cite{va88})	
\end{remark}

\subsection{~}As one can see from the proof, the only obstruction for proving the cohomological tautological conjecture of all hyperk\"ahler is that we do not know whether the $i$-th component of  $c_k(\cT_\pi)$ in the decomposition \eqref{Dec} are theta classes when $\frac{m_\Lambda-3}{2}\leq i\leq \min\{2k,\frac{3(m_\Lambda-3)}{2}\}$. 	

 As pointed out in \cite{Vo12} and \cite{V1,SV16}, the Beauville-Voisin conjecture actually predicts that all $i$-th component of $c_k(\cT_\pi)$ lies in the last component (after possibly shrinking to an open subset of $\cF_h$). This can be deduced from the existence of so called ``Chow-Kunneth" decomposition (cf.~\cite{Mu93}).  Therefore, the Beauville-Voisin conjecture implies that Conjecture \ref{conj5} is true at least after shrinking to an open subset of $\cF_h$.

 %%%%%% old proof. When the very general fiber of $\cF_h$ has non-trivial automorphisms, one needs to take a finite covering of $\cF_h$ and apply Lemma \ref{Descend2}. All the argument goes through except the existence of a lift of $c_j(\cT_\pi)$ to $\bigoplus\limits_{p+q=2j} H^p(Y_K,\bH^q).$ As the classes of the line bundles $c_1(\cL_i)$ can still be lifted to $\bigoplus\limits_{p+q=2} H^p(Y_K,\bH^q)$, we can restrict to special $\kappa$-classes and obtain a slightly weaker result as below. \begin{theorem}\label{tauthm2}With the notations as in Theorem \ref{tauthm} and assume that $b\geq 3$, the special $\kappa$-classes $[\kappa_{a_1\ldots, a_r}^\cB]\in H^{\ast}(\cF,\QQ)$ is lying in $\NL^\ast_{\rm hom}(\cF)$. \end{theorem} 

\subsection{Further remarks} 
 For polarized hyper\"ahler manifolds of generalized Kummer type, we know that the moduli space $\cF_h$ is 4-dimensional.  In this case, the $\kappa$-classes on moduli space of lattice-polarized hyperk\"ahler manifolds automatically maps to $\NL^\ast_{\hom}(\cF_h)$ via the pushforward map. This is because  $\CH^1(\cF_h)$ is spanned by NL-divisors and $R^3_{\hom}(\cF_h)=R^4_{\hom}(\cF_h)=0$. So we have   
$$\NL^i_{\hom}=R^i_{\hom}(\cF_h), \hbox{ for all $i\neq2$}.$$  Then it suffices  to check whether the  $\kappa$-class 
\begin{equation}\label{egk}
 [\kappa_{a;b_1,\ldots,b_{2n}}]\in H^4(\cF_h,\QQ), \end{equation}  (i.e.  $a+\sum\limits_{j=1}^{2n} jb_j-2n=2$)   is lying in $\NL^2_{\hom}(\cF_h)$. By Theorem \ref{tauthm}, we know this is true when $b_i=0$ for $i>1$. When $b_i\neq 0$ for some $i>1$, one can see some examples computed in  $\S$\ref{GRR}. In the case of $n=2$,  combining these results together, we know that the conjecture holds if and only if  $$\kappa_{0;0,3,0,0}, \kappa_{1;0,1,1,0}, \kappa_{2;0,2,0,0},\kappa_{2;0,0,0,1},\kappa_{3;0,0,1,0}, \kappa_{4;0,1,0,0}$$  are contained in  $\NL^2_{\hom}(\cF_h)$.

\section{Ring of Special cycles}
In Kudla's program,  it is  more natural to consider the weighted cycles on Shimura varieties instead of the connected cycles. These are so called special cycles. In this section, we discuss the  properties of the special cycles and their applications. 

\subsection{Kudla's special cycles} Keep the same notations as in \S \ref{sc}. We let $\lambda$ be the first Chern class of Hodge line bundle on $Y_K$.  
For any $\beta \in \mathrm{Sym}^r(\QQ)$, we set 
 $$\Omega_\beta=\{ \bv\in V^n: \frac{1}{2} (\bv,\bv)=\beta, \ \dim U(\bv)=\rank \beta\},$$
and  the special cycles on $Y_\Gamma$ of codimension $r$ are defined as 
 \begin{equation} \label{specialcycle}
 Z(\beta, \varphi,K)=\lambda^{r-\rank\beta}\cdot \sum\limits_{\bv\in \Omega_\beta \atop\mod\Gamma_1} \varphi(\bv) c(U(\bv),1,K),
 \end{equation}
 with $\varphi$ being a $K$-invariant Schwartz function on $V^r(\AA_f)$. Then $Z(\beta,\varphi,K)$ can be viewed as a cycle class in $\CH^r(Y_K)$ if $\rank \beta=r$.
  As before, we let $\widetilde{\SC}^r(Y_K)$ be the subspace of $\CH^r(Y_K)$  spanned by  $ Z(\beta, \varphi,K)$ and $$\widetilde{\mathrm{SC}}^\ast(Y_K)\subseteq \CH^\ast(Y_K)$$  the subring generated by  all special cycles.    
 The special cycles have many nice properties, e.g. it behaves  well under pullback (cf.~\cite{Ku97}). More importantly, although in general  the  Chow ring $\CH^\ast(Y_K)$ is  not finitely generated, we have
 \begin{theorem}
 	The ring $\widetilde{\SC}^\ast(Y_K)$ is finitely generated and  its image in $H^\ast(Y_K,\QQ)$ contains  $H_\theta^\ast(Y_K,\QQ)$. 
 \end{theorem}
The first statement follows from the work of Zhang {\cite{ZW09}} and a recent result of Bruinier and Westerholt-Raum \cite{BWR15}, while the second assertion is proved in  \cite{BLMM16}. Clearly, there is  an inclusion  $$\widetilde{\SC}^\ast(Y_K)\subseteq\SC^\ast(Y_K). $$  Moreover, we believe that  the following statement should be true. 
\begin{conjecture}
	$\SC^\ast(Y_K)=\widetilde{\SC}^\ast(Y_K)$. 
\end{conjecture}

We shall remark that this conjecture can be easily checked in some simple cases such as the moduli space of polarized K3 surfaces, or the moduli space of some polarized $K3^{[n]}$-type hyperk\"ahler manifolds (See Proposition \ref{K3example}). Recall that the second  cohomology of a K3 surface is isomorphic to the even unimodular lattice 
$$L_{K3}=U^{\oplus 2}\oplus E_8^{\oplus 3}(-1),$$
and the primitive part of a polarized K3 surface $(X,H)$ in $\cK_g$
is isomorphic to  $$L_g:=v^\perp=(2-2g)\oplus U^{\oplus 2}\oplus E_8^{\oplus 2}(-1),$$
where $v\in L_{K3}$ represents the class $c_1(H)$.    The moduli space $\cK_g$  of polarized K3 surfaces of genus $g$ is an open subset of the arithmetic quotient $Y_{L_g}=\Gamma_{L_g}\backslash D$, where $\Gamma_{L_g}$ is the associated polarized monodromy group.  In this situation, we have

\begin{proposition} \label{K3example}
	$\SC^r(Y_{L_g})=\widetilde{\SC}^r(Y_{L_g})$ for all $r$. 
\end{proposition}
\begin{proof} 
	The arithmetic subgroup $\Gamma_{L_g}$ is nothing but the collection of isometries of $L_{K_3}$ preserving $v$.  Every connected cycle $c(U)$ in $Y_{L_g}$ of codimension $r$ can be described as the map  $$\Gamma_\Sigma\backslash D_\Sigma\rightarrow Y_{L_g}, $$ 
	where $\Sigma=\mathrm{span}\{U,v\}\cap L_{K3}$ is a primitive sublattice of $L_{K3}$ and $D_\Sigma$ and $\Gamma_\Sigma$ are defined as in \S\ref{Slattice}.  We write  $c(\Sigma)=c(U)$ for this connected cycle. 
	
	As $L_{K_3}$ is unimodular,  every even lattice of signature $(1,r)$ ($r\leq 19$) admits a unique (up to isometry) primitive embedding to  $L_{K_3}$. 
	Let $\Sigma$ be an even lattice of signature $(1,r)$. 	 The special cycle on $Y_{L_g}$ is of the form  $$Z(\Sigma)=\sum\limits_{\substack{\widehat{\phi(\Sigma)}=\Sigma' \\\mod \Gamma_{L_g}}}c(\Sigma'),$$ 
	where the sum  runs over all embeddings $\phi:\Sigma\hookrightarrow L_{K3}$ modulo the isometries in $\Gamma_{L_g}$ and $\widehat{\phi(\Sigma)}$ is the saturation of $\phi(\Sigma)$ in $L_{K3}$. We can write $$Z(\Sigma)=c(\Sigma)+\sum\limits_{\substack{\Sigma\ncong\Sigma' \\ \Sigma\hookrightarrow\Sigma'}}  m(\Sigma')c(\Sigma')$$ for some integers $m(\Sigma')>0$. 
	Then it is easy to see that $c(\Sigma)$ is in the span of $Z(\Sigma')$  and  this proves the assertion.

\end{proof}

\subsection{Dimension of $\widetilde{\SC}^r(Y_K)$} Let us give some examples to explain  how to  use Zhang-Bruinier-Raum's result to construct relations between special cycles and estimate the dimension of $\widetilde{\SC}^i(Y_K)$. Let $\Lambda$ be an even lattice of signature $(2,b)$.  For simplicity, we restrict ourselves to the case $Y_K$ is the arithmetic quotient $Y_\Lambda:=\Gamma_\Lambda\backslash D$, where $\Gamma_\Lambda$ is the arithmetic subgroup  
\begin{equation}
\{g\in \SO(\Lambda)|~\hbox{$g$ acts trivially on $d(\Lambda)$}\}.
\end{equation}
 Let $\HH_r$ be the  Siegel upper half-plane of genus $r$. The metaplectic double cover $\Mp_{2r}(\ZZ)$ of $\Sp_{2r}(\ZZ)$ consists of
pairs $\left(M, \phi(\tau)\right)$,  where $$ M=
\left(\begin{array}{cc}a & b \\c & d\end{array}\right)\in \Sp_{2r}(\ZZ),
~~\phi: \HH_r\rightarrow \CC ~\hbox{with}~\phi^2(\tau)=\det(c\tau+d). $$ 
Let $\rho_\Lambda^{(r)}$ be the Weil-representation of $\Mp_{2r}(\ZZ)$ on $\CC[d(\Lambda)^r]$. For
any $k\in \frac{1}{2} \ZZ$,  a vector-valued Siegel modular form $f(\tau)$
of weight $k$ and type $\rho_\Lambda^{(r)}$  is a  $\CC[d(\Lambda)^r]$-valued holomorphic function on
$\HH_r$, such that
$$f(M\tau)=\phi(\tau)^{2k}\cdot \rho_\Lambda^{(r)}(M,\phi)
(f),~\textrm{for all} ~(M,\phi)\in \Mp_{2r}(\ZZ)$$
and it is a Siegel cusp form if $f(\tau)$ is holomorphic at all cusps. We  denote by  $M^{r}_{k,\Lambda}(\rho_\Lambda)$ (resp. $S^{r}_{k,\Lambda}(\rho_\Lambda)$) the space of  $\CC[d(\Lambda)^r]$-valued modular (resp. cusp) forms of weight $k$ and type $\rho_\Lambda^{(r)}$. Then  
   
\begin{proposition} \label{estimate}
For $1\leq r\leq b$, we have
	\begin{equation}\label{dim1}
	\dim \widetilde{\SC}^r(Y_\Lambda)\leq \dim \widetilde{\SC}^{r-1}(Y_\Lambda)+\dim S_{\frac{b+2}{2},M}^{r}(\rho_\Lambda^\vee).
	\end{equation}	 
\end{proposition}
\begin{proof} The proof is similar as the case $r=1$ proved in \cite[\S5]{Br02}. 
The work of Zhang-Bruinier-Raum says that the generating series 
	\begin{equation}
\Phi(\tau)=\sum Z(\beta,\varphi) q^{\beta} \varphi^\ast, 
	\end{equation}
is a $\CC[d(\Lambda)^r]$-valued Siegel modular form of weight $\frac{b+2}{2}$ and  type $\rho_{\Lambda}$, where $q^\beta=\exp(2\pi {\rm Tr}(\beta \tau))$ (cf.~\cite[Theorem 5.2]{BWR15}).   Using the Petersson inner product  
$$M^{r}_{\frac{b+2}{2}, \Lambda}(\rho_\Lambda)\times S^{r}_{\frac{b+2}{2}, \Lambda}(\rho_\Lambda^\vee)\rightarrow \CC, $$
we obtain a natural map
\begin{equation}\label{mod}
S^{r}_{\frac{b+2}{2},\Lambda} (\rho_\Lambda^\vee)\rightarrow \widetilde{\SC}^r(Y_\Lambda)/ \lambda\cdot \widetilde{\SC}^{r-1}(Y_\Lambda) 
\end{equation}	
by tensoring with $\Phi(\tau)$. The assertion then follows directly from the surjectivity of \eqref{mod}.
\end{proof}

For general $Y_K$, the statement is still true once we replace  $S_{\frac{b+2}{2},M}^{r}(\rho_\Lambda^\vee)$ by the space of cuspidale vector-valued modular forms with respect to certain congruence subgroup of $\Mp_{2r}(\ZZ)$.  This also follows from the modularity result of Zhang-Bruinier-Raum.

\begin{remark}
	The inequality \eqref{dim1} is in general strict (see \cite[\S6.1]{Br14}). A very interesting question to investigate is when the equality holds (see~\cite[Theorem 2]{Br14}). This would enable us to compute the Betti numbers of $Y_K$ and also the dimension of $\widetilde{\SC}^r(Y_K)$.
\end{remark}

\subsection{Further questions}Let us step back to $\SC^r(Y_K)$.  
Similar as the result in \cite{GT05}, when $b\geq 3$, one can show that $$\lambda^{b-2}\in \mathrm{SC}^{b-2}(Y_K)\neq 0$$ while $\lambda^{b-1}=\lambda^{b}=0$ in $\CH^\ast(Y_K)$. Furthermore, we have 
\begin{proposition} For $i>b-2$ and $j>1$, the subgroup $$\lambda^j\cdot \mathrm{SC}^{i-j}(Y_K)\subseteq \mathrm{SC}^{i}(Y_K)$$ is zero. 
\end{proposition} 

\begin{proof}
	The proof is the same as in \cite{GT05}, where one can show that the self intersection of Hodge line bundle of codimension $\geq b-1$  on a Shimura variety of orthogonal type  with dimension $b\geq 3$ is rationally equivalently to zero. 
\end{proof}

Recall that  $\SC^r_{\hom}(Y_K)=0$ for $r>b-2$ by Corollary \ref{Cpet}, we  wonder if the questions below have a positive answer. 
\begin{question}\label{que}
 	Assume that $\dim Y_K\geq 3$. Do the following hold?
 	\begin{enumerate}
 		\item  $\mathrm{SC}^{r} (Y_K)=0 $ for $r>b-2$.
 		\item 	The cycle class map induces an {\rm isomorphism} $
 		cl:\mathrm{SC}^\ast(Y_K)\rightarrow \mathrm{SC}^\ast_{\rm hom}(Y_K)
 		$. 
 	\end{enumerate}
 \end{question}

Let us come back to $\cK_g$, where the first part of Question \ref{que} indicates that  $\NL^r(\cK_g)=0$ for $r>17$. This can be viewed as an analogous of the  Faber's conjecture on $\cM_g$. The second part is then equivalent to saying that there is an isomorphism $$\NL^\ast(\cK_g)\cong \NL^\ast_{\rm hom}(\cK_g). $$
For small $g$, one can get  partial results towards these questions via the geometric construction of $\cK_g$.  
\begin{proposition}When $g\leq 14$ or $g\in\{ 16,18,20\} $,  $\NL^{19}(\cK_g)=0$.  \end{proposition}\begin{proof} The first assertion follows from the unirationality of $\cK_g$ for small $g$.	Indeed, Mukai has proved that $\cK_g$ is unirational when $g\leq 13$ or $g\in \{16,18,20\}$ and Nuer \cite{Nu15} recently proved that $\cK_{14}$ is also unirational.  \end{proof}

\begin{remark}In this case, as the Betti number of $\cK_2$ has been computed in \cite{KL89}, one can  make explicity computation to verify Question \ref{que} (2) via Proposition \ref{estimate}.
\end{remark}

\bibliographystyle {plain}
\bibliography{Universal}

\end{document}